\documentclass[a4paper,11pt]{amsart}
\usepackage{amssymb,amscd,amsxtra,psfrag,xypic}
\usepackage[dvipdfmx]{graphicx}
\usepackage{color}

\usepackage[pdftex]{hyperref}
\hypersetup{%
bookmarks=true
bookmarksnumbered=true,%
colorlinks=true,%
setpagesize=false,%
pdftitle={},%
pdfauthor={Takuzo Okada}%
pdfsubject={}%
pdfkeywords={}}

\theoremstyle{plain}
\newtheorem{Thm}{Theorem}[section]
\newtheorem{Lem}[Thm]{Lemma}
\newtheorem{Cor}[Thm]{Corollary}
\newtheorem{Prop}[Thm]{Proposition}

\theoremstyle{definition}
\newtheorem{Def}[Thm]{Definition}
\newtheorem{Def-Lem}[Thm]{Definition-Lemma}
\newtheorem{Cond}[Thm]{Condition}
\newtheorem{Rem}[Thm]{Remark}

\newtheorem*{Ack}{Acknowledgments}

\newtheorem{Setting}[Thm]{Setting}

\theoremstyle{remark}

\newcommand{\prt}{\partial}
\newcommand{\rank}{\operatorname{rank}}

\newcommand{\Sing}{\operatorname{Sing}}
\newcommand{\Spec}{\operatorname{Spec}}

\newcommand{\Crit}{\operatorname{Crit}}
\newcommand{\Cox}{\operatorname{Cox}}

\newcommand{\mbA}{\mathbb{A}}

\newcommand{\mbC}{\mathbb{C}}
\newcommand{\mbF}{\mathbb{F}}
\newcommand{\mbG}{\mathbb{G}}

\newcommand{\mbP}{\mathbb{P}}

\newcommand{\mbZ}{\mathbb{Z}}

\newcommand{\mcE}{\mathcal{E}}

\newcommand{\mcL}{\mathcal{L}}
\newcommand{\mcM}{\mathcal{M}}
\newcommand{\mcN}{\mathcal{N}}
\newcommand{\mcO}{\mathcal{O}}

\newcommand{\mcQ}{\mathcal{Q}}

\newcommand{\mcX}{\mathcal{X}}

\newcommand{\mfm}{\mathfrak{m}}

\newcommand{\mfp}{\mathfrak{p}}

\newcommand{\msp}{\mathsf{p}}
\newcommand{\msq}{\mathsf{q}}
\newcommand{\K}{\Bbbk}

\newcommand{\inj}{\hookrightarrow}

\newcommand{\ratmap}{\dashrightarrow}

\newcommand{\CH}{\operatorname{CH}}

\newcommand{\chara}{\operatorname{char}}

\title[Stable rationality of del Pezzo fibrations]{Stable rationality of del Pezzo fibrations of low degree over projective spaces}
\author{Igor~Krylov \and Takuzo~Okada}
\address{Max Planck Institute for Mathematics, Vivatsgasse 7, Bonn 53111 Germany}
\email{igor@krylov.su}
\address{Department of Mathematics, Faculty of Science and Engineering, Saga University, Saga 840-8502 Japan}
\email{okada@cc.saga-u.ac.jp}
\date{}

\begin{document}

\begin{abstract}
The main aim of this article is to show that a very general $3$-dimensional del Pezzo fibration of degree 1,2,3 is not stably rational except for a del Pezzo fibration of degree 3 belonging to explicitly described 2 families.
Higher dimensional generalizations are also discussed and we prove that a very general del Pezzo fibration of degree 1,2,3 defined over the projective space is not stably rational provided that the anticanonical divisor is not ample.
\end{abstract}

\maketitle


\section{Introduction} \label{sec:intro}

In this paper we study stable rationality of del Pezzo fibrations of degrees $1$, $2$, and $3$. 
Recent breakthroughs by Voisin \cite{Voisin}, expanded by Colliot-Th\'el\`ene and Pirutka \cite{CTP}, have changed the landscape of the study of stable rationality. 
Consequently failure of stable rationality was proven for large classes of rationally connected varieties (see e.g.\ \cite{Totaro} for hypersurfaces and \cite{AO} for conic bundles both in arbitrary dimension).
For many families of varieties of dimension $\geqslant 4$ in these classes, even rationality was not known.
We apply the techniques of Colliot-Th\'el\`ene and Pirutka to del Pezzo fibrations: a subclass of Mori fiber spaces.

\begin{Thm} \label{mainthm1}
Let $X$ be a very general $n$-dimensional nonsingular del Pezzo fibration of degree $1$, $2$, or $3$ over $\mbP^{n-2}$ embedded as a hypersurface in a toric $\mbP (1,1,2,3)$-, $\mbP (1,1,1,2)$-, or $\mbP^3$-bundle over $\mbP^{n-2}$, respectively, where $n \ge 3$.
If $-K_X$ is not ample, then $X$ is not stably rational.
\end{Thm} 

We use the Chow group of zero cycles $\CH_0(X)$ to detect stable rationality. 
If a variety is not universally $\CH_0$-trivial, then it is not stably rational.
To prove Theorem \ref{mainthm1} we find a non universally $\CH_0$-trivial reduction $\overline{X}$ of $X$ to a finite characteristic.
To do this we use Koll\'ar's technique to show that a reduction to characteristic $2$ for del Pezzo fibrations of degrees $1,2$ and to characteristic $3$ (or $2$) for del Pezzo fibrations of degree $3$ is not universally $\CH_0$-trivial under a suitable condition such as the ampleness of the anti-canonical bundle.
We then use the specialization theorem \cite[Th\'eor\`eme~1.14]{CTP} of Coliot-Th\'el\`ene and Pirutka to lift the results back to characteristic $0$.

Theorem \ref{mainthm1} is not in the strongest form of the result of this paper: we obtain stronger results in Theorems \ref{thm:dP1}, \ref{thm:dP2} and \ref{thm:dP3}, whose statements require some preparations.
These cover the following varieties as a special case.

\begin{Thm} \label{mainthm3}
Suppose that $n \ge 3$.
The following varieties are not stably rational.
\begin{enumerate}
\item A double cover of $\mbP^{n-2} \times \mbP^2$ branched along a very general divisor of bi-degree $(2 m,4)$ for $m \ge (n-1)/2$.
\item A triple cover of $\mbP^{n-2} \times \mbP^2$ branched along a very general divisor of bi-degree $(3 m,3)$ for $m \ge (n-1)/3$.
\item A very general hypersurface of bi-degree $(d,3)$ in $\mbP^{n-2} \times \mbP^3$ for $d \ge n-1$.
\end{enumerate}
\end{Thm}

Note that the variety in (1) (resp.\ (2) and (3)), together with the morphism to the first factor $\mbP^{n-2}$, is a del Pezzo fibration of degree $2$ (resp.\ degree $3$).
The results (2) and (3) improve the corresponding results of \cite{Kollar2}. 

It should be pointed out that Theorem \ref{mainthm1} or even the above mentioned stronger versions do not cover many families which we expect are not stably rational when $n \ge 4$: in that case a nonsingular del Pezzo fibration of degree $1,2$ or $3$ over $\mbP^{n-2}$ is not necessarily  embedded in a toric $\mbP (1,1,2,3)$-, $\mbP (1,1,1,2)$-, or $\mbP^3$-bundle over $\mbP^{n-2}$.

Let us turn our attention to the $3$-dimensional case.
It is known that a nonsingular del Pezzo fibration $X/\mbP^1$ of degree $1,2,3$ is birationally rigid if $-K_X$ is not in the interior of the mobile cone, which implies non-rationality of $X$ (cf.\ \cite{Puk, Gri1, Gri2}).
Moreover it is proved in \cite{C08} that a very general nonsingular del Pezzo fibration $X/\mbP^1$ of degree 3 is rational if and only if $X$ is a hypersurface of bi-degree $(1,3)$ in $\mbP^1 \times \mbP^3$.
The results of this paper, Theorems \ref{thm:dP1}, \ref{thm:dP2}, \ref{thm:dP3}, in dimension three are very satisfying.
They cover all families of nonsingular del Pezzo fibrations of degree $1$, $2$ embedded in a toric $\mbP (1,1,2,3)$-, $\mbP (1,1,1,2)$ bundle over $\mbP^1$, respectively, and all but two families of nonsingular del Pezzo fibrations of degree $3$ embedded in a toric $\mbP^3$-bundle over $\mbP^1$.
Moreover $3$-dimensional nonsingular del Pezzo fibrations of degrees $1,2,3$ can always be embedded into a toric $\mbP (1,1,2,3)$-, $\mbP (1,1,1,2)$-, or $\mbP^3$-bundle over $\mbP^1$ (see e.g.\ \cite{Puk}).
Thus we get the following theorem.

\begin{Thm} \label{mainthm2}
A very general $3$-dimensional nonsingular del Pezzo fibration of degree $1,2,3$ is not stably rational except when $X$ is a del Pezzo fibration of degree $3$ and belongs to one of the explicitly described two families.
\end{Thm}

In the above exceptions, one family is the family of hypersurfaces of bi-degree $(1,3)$ in $\mbP^1 \times \mbP^3$ and the other is the family obtained by blowing-up cubic $3$-folds along a smooth plane cubic curve.
A general member of the former and the latter family is rational and non-rational (\cite{CG}), respectively.

Rationality questions for del fibrations of degree $4$ over $\mbP^1$ are settled in \cite{Alekseev,Shramov} and it is shown in \cite{HT} that a very general del Pezzo fibration over $\mbP^1$ of degree $4$ which is not rational and not birational to a cubic threefold is not stably rational.
It is also well known that del Pezzo fibrations over $\mbP^1$ of degree $\geqslant 5$ are rational. 
Combining these results with Theorem \ref{mainthm2} we get the following.

\begin{Thm}
Let $\pi:X\to \mbP^1$ be a very general del Pezzo fibration. Suppose $X$ is smooth and not birational to a cubic threefold. Then $X$ is either rational or not stably rational.
\end{Thm}

Thus in dimension $3$, as far as very general members are concerned, the study of stable (non-)rationality is completely settled (modulo cubic $3$-folds) for del Pezzo fibrations.
Earlier Hassett and Tschinkel have proven a similar result for Fano varieties. They have shown that a very general smooth Fano variety which is not birational to a smooth cubic threefolds is either rational or not stably rational \cite{HT} (see also \cite{Okada2} for the similar result for orbifold Fano $3$-fold hypersurfaces).

\begin{Ack}
The second author is partially supported by JSPS KAKENHI Grant Number 26800019.
\end{Ack}

\section{Preliminaries}

\subsection{Weighted projective space bundles}

In this section we assume that the ground field is an algebraically closed field $k$.

\subsubsection{Definition of WPS bundles}
A {\it toric weighted projective space bundle over $\mbP^n$} is a projective simplicial toric variety $P$ with Cox ring
\[
\Cox (P) = k [u_0,\dots,u_n,x_0,\dots,x_m],
\]
which is $\mbZ^2$-graded as
\[
\begin{pmatrix}
1 & \cdots & 1 & \lambda_0 & \cdots & \lambda_m \\
0 & \cdots & 0 & a_0 & \cdots & a_m
\end{pmatrix}
\]
and with the irrelevant ideal $I = (u_0,\dots,u_n) \cap (x_0,\dots,x_m)$, where $\lambda_0,\dots,\lambda_m$ are integers and $n, m, a_0,\dots,a_m$ are positive integers.
In other words, $P$ is the geometric quotient
\[
P = (\mbA^{n+m+2} \setminus V (I))/\mbG_m^2,
\]
where the action of $\mbG_m^2 = \mbG_m \times \mbG_m$ on $\mbA^{n+m+2} = \Spec k [u_0,\dots,u_n,x_0,\dots,x_m]$ is given by the above matrix.
We will simply say that $P$ is the WPS bundle over $\mbP^n$ defined by
\[
\begin{pmatrix}
u_0 & \cdots & u_n & & x_0 & \cdots & x_m \\
1 & \cdots & 1 & | & \lambda_0 & \cdots & \lambda_m \\
0 & \cdots & 0 & | & a_0 & \cdots & a_m
\end{pmatrix}.
\]
There is a natural projection $P \to \mbP^n$, which is the projection by the coordinates $u_0,\dots,u_n$, and its fiber is isomorphic to the weighted projective space $\mbP (a_0,\dots,a_m)$.
We also call $P$ (or $P \to \mbP^n$) a {\it $\mbP (a_0,\dots,a_m)$-bundle over $\mbP^n$}.

Let $\msp \in P$ be a point and let $\msq \in \mbA^{n+m+2} \setminus V (I)$ be a preimage of $\msp$ via the morphism $\mbA^{n+m+2} \setminus V (I) \to P$ and write $\msq = (\alpha_0,\dots,\alpha_n,\beta_0,\dots,\beta_m)$.
In this case we express $\msp \in P$ as $\msp = (\alpha_0\!:\!\cdots\!:\!\alpha_n ; \beta_0\!:\!\cdots\!:\!\beta_n)$.
This is clearly independent of the choice of $\msq$.

We will frequently replace coordinates in order to simplify the expression of a given point $\msp \in P$.
Consider a point $\msp = (\alpha_0\!:\!\cdots\!:\!\alpha_n ; \beta_0\!:\!\cdots\!:\!\beta_n)$ and suppose for example that $\alpha_0 \ne 0$, $\beta_j \ne 0$, $a_j = 1$.
Then for $l \ne j$ such that $\lambda_l/a_l \ge \lambda_j$, the replacement 
\[
x_l \mapsto \alpha_0^{\lambda_l - a_l \lambda_j} \beta_j^{a_l} x_l - \beta_l u_0^{\lambda_l - a_l \lambda_j} x_j^{a_l}
\] 
induces an automorphism of $P$.
By considering this coordinate change, we may assume that the $x_l$-coordinate is zero for $l$ such that $\lambda_l/a_l \ge \lambda_j$.

\subsubsection{Weil divisors}

Let $P$ be a $\mbP (a_0,\dots,a_m)$-bundle as above.
The action of $\mbZ^2$ on $\Cox (P)$ is given in the above matrix.
We have the decomposition
\[
\Cox (P) = \bigoplus_{(\alpha,\beta) \in \mbZ^2} \Cox (P)_{(\alpha,\beta)},
\]
where $\Cox (P)_{(\alpha,\beta)}$ consists of the homogeneous elements of bi-degree $(\alpha,\beta)$.
An element $f \in \Cox (P)_{(\alpha,\beta)}$ is called a (homogeneous) polynomial of bi-degree $(\alpha,\beta)$.

The (Weil) divisor class group $\operatorname{Cl} (P)$ of $P$ is isomorphic to $\mbZ^2$.
Let $F$ and $D$ be the divisors on $P$ corresponding to $(1,0)$ and $(0,1)$, respectively, which form generators of $\operatorname{Cl} (P)$.
Then $F$ is the class of a fiber of the pullback of a hyperplane on $\mbP^n$ and the zero locus $(x_i = 0)$ is linearly equivalent to $\lambda_i F + a_i D$.
We denote by $\mcO_P (\alpha,\beta)$ the rank $1$ reflexive sheaf corresponding to the divisor class of type $(\alpha,\beta)$.
More generally, for a subscheme $Z \subset P$, we set $\mcO_Z (\alpha,\beta) = \mcO_X (\alpha,\beta)|_Z$.
Finally we remark that there is an isomorphism
\[
H^0 (P, \mcO_P (\alpha,\beta)) \cong \Cox (P)_{(\alpha,\beta)}
\]
and that the cone of nef divisors on $P$ is generated by $F$ and $\lambda_l F + a_l D$, where $l$ is such that $\lambda_l/a_l = \max \{ \lambda_j/a_j \mid j = 0,\dots,m\}$.

\subsubsection{Affine charts}

We give a description of standard open affine charts of $P$.
For $i = 0,\dots,n$ and $j = 0,\dots,m$, we define $U_{i,j} = (u_i \ne 0) \cap (x_j \ne 0) \subset P$.
Clearly the $U_{i,j}$ cover $P$.
We only explain an explicit description of $U_{i,j}$ for $j$ such that $a_i = 1$, which is enough for our purpose.
For $k \ne i$ and $l \ne j$, we set 
\[
u_k^{(i,j)} = \frac{u_k}{u_i} \quad \text{and} \quad x_l^{(i,j)} = \frac{u_i^{a_l \lambda_j - \lambda_l} x_l}{x_j^{a_l}},
\]
which are clearly $\mbG_m^2$-invariant rational functions on $P$ which are regular on $U_{i,j}$.
Moreover it is easy to see that $U_{i,j}$ is isomorphic to the affine $(n+m)$-space with affine coordinates $\{u_k^{(i,j)} \mid k \ne i\} \cup \{x_l^{(i,j)} \mid l \ne j\}$.
In the following, we abuse the notation and we identify $u_k^{(i,j)}$ with $u_j$ and $x_l^{(i,j)}$ with $x_l$, and we think of $U_{i,j}$ as affine space with coordinates $\{u_0,\dots,u_n,x_0,\dots,x_m\} \setminus \{u_i,x_j\}$.
Under this terminology, the restriction map
\[
H^0 (P, \mcO_P (\alpha,\beta)) \to H^0 (U_{i,j}, \mcO_{P} (\alpha,\beta)) \cong H^0 (\mbA^{n+m}, \mcO_{\mbA^{n+m}})
\]
can be understood as a homomorphism defined by substituting $u_i = x_j = 1$ in $g (u,x) \in H^0 (P,\mcO_P (\alpha,\beta))$.

\subsection{Universal $\CH_0$-triviality}

For a variety $X$, we denote by $\CH_0 (X)$ the {\it Chow group of $0$-cycles} on $X$.

\begin{Def}
\begin{enumerate}
\item A projective variety $X$ defined over a field $k$ is {\it universally $\CH_0$-trivial} if for any field $F$ containing $k$, the degree map $\CH_0 (X_F) \to \mbZ$ is an isomorphism.
\item A projective morphism $\varphi \colon Y \to X$ defined over a field $k$ is {\it universally $\CH_0$-trivial} if for any field $F$ containing $k$, the push-forward map $\varphi_* \colon \CH_0 (Y_F) \to \CH_0 (X_F)$ is an isomorphism.
\end{enumerate}
\end{Def}

We apply the specialization argument of universal $\CH_0$-triviality in the following form.

\begin{Thm}[{\cite[Th\'eor\`eme 1.14]{CTP}}] \label{thm:spCH0}
Let $A$ be a discrete valuation ring with fraction field $K$ and residue field $k$, with $k$ algebraically closed.
Let $\mcX$ be a flat proper scheme over $A$ with geometrically integral fibers.
Let $X$ be the generic fiber $\mcX \times_A K$ and $Y$ the special fiber $\mcX \times_A k$.
Assume that $Y$ admits a universally $\CH_0$-trivial resolution $\tilde{Y} \to Y$ of singularities.
Let $\overline{K}$ be an algebraic closure of $K$ and assume that the geometric generic fiber $X_{\overline{K}}$ admits a resolution $\tilde{X} \to X_{\overline{K}}$.
If $\tilde{X}$ is universally $\CH_0$-trivial, then so is $\tilde{Y}$.
\end{Thm}

In our applications of Theorem \ref{thm:spCH0}, the variety $X_{\overline{K}}$ is  always smooth.

\begin{Lem}[{\cite[Lemma 2.2]{Totaro}}] \label{lem:totaro}
Let $X$ be a smooth projective variety over a field.
If $H^0 (X,\Omega_X^i) \ne 0$ for some $i > 0$, then $X$ is not universally $\CH_0$-trivial.
\end{Lem}

\subsection{Cyclic covers in positive characteristic} \label{sec:cyccov}

We briefly recall Koll\'ar's technique of constructing a suitable invertible sheaf on cyclic covers in positive characteristics.

Let $Z$ be a smooth variety of dimension $n$ over an algebraically closed field of characteristic $p > 0$, $\mcL$ an invertible sheaf on $Z$ and $s \in H^0 (Z,\mcL^p)$ a global section.
Let $\tau \colon X \to Z$ be the cyclic cover of degree $p$ branched along the zero locus $(s = 0) \subset Z$.
By \cite[V.5]{Kollar}, there exists an invertible sheaf $\mcQ$ on $Z$ such that $\tau^*\mcQ \subset (\Omega_X^{n-1})^{\vee \vee}$, where $\vee \vee$ denotes the double dual.
We set $\mcM := \tau^*\mcQ$ and call it the {\it invertible sheaf associated with the covering $\tau$}.   
Note that if the branched divisor $(s = 0)$ is reduced, then $\mcQ \cong \omega_Z \otimes \mcL^p$ by \cite[Lemma V.5.9]{Kollar}.

We recall the definition of critical point of $s \in H^0 (Z, \mcL^p)$ which plays an important role in the analysis of singularities of $X$.
Let $\msp \in Z$ be a point and let $x_1,\dots,x_n$ be local coordinates of $Z$ at $\msp$.
Take a local generator $\mu$ of $\mcL$ at $\msp$ and write $s = f \mu^p$, where $f = f (x_1,\dots,x_n)$.
We write $f = f_0 + f_1 + \cdots$, where $f_i$ is homogeneous of degree $i$ and we set
\[
H (s) = \left( \frac{\prt^2 f}{\prt x_i \prt x_j} \right).
\]

\begin{Def}
We say that $s$ has a {\it critical point} at $\msp$ if $f_1 = 0$.
Suppose that $s$ has a critical point at $\msp$.
We say that $s$ has a {\it nondegenerate critical point} at $\msp$ if $\rank H (s) (\msp) = n$.
When $p = 2$ and $n$ is odd, we always have $\rank H (s) (\msp) < n$.
In this case, we say that $s$ has an {\it almost nondegenerate critical point} at $\msp$ if 
\[
\operatorname{length} \mcO_{Z,\msp}/(\prt f/\prt x_1,\dots,\prt f/\prt x_n) = 2.
\]
\end{Def}

The above definition does not depend on the choice of local coordinates $x_1,\dots,x_n$ and the local generator $\mu$ of $\mcL$ (see \cite[Section V.5]{Kollar}).

\begin{Rem} \label{rem:critunit}
Let $s = f \mu^p$ be as above and $a \in \mcO_{Z,\msp}$ an invertible element.
We write $a = a (x_1,\dots,x_n) = a_0 + a_1 + \cdots$ as before.
We think of $a^p s$ as a section around $\msp$ and compare critical points of $s$ and $a^p s$ at $\msp$.
It is obvious that $s$ has a critical point at $\msp$ if and only if so does $a^p s$ since the linear term of $a f$ is $a_0 f_1$ and $a_0 \ne 0$ (Here recall that that the ground field is of characteristic $p$).
Now suppose that $s$ and $a^p s$ has a critical point at $\msp$.
Then it is also easy to see that $s$ has a nondegenerate critical point (resp.\ (almost) nondegenerate critical point) at $\msp$ if and only if so does $a^p s$ since $H (s) = a^p H (s)$ and 
\[
\mcO_{Z,\msp}/(\prt (a^p f)/\prt x_1,\dots,\prt (a^p f)/\prt x_n)
= \mcO_{Z,\msp}/(\prt f/\prt x_1,\dots,\prt f/\prt x_n).
\]
This observation will be used later.
\end{Rem}

\begin{Rem}
We explain an explicit description of (almost) nondegenerate critical point.
We refer readers to \cite[Section V.5]{Kollar} for details.
Suppose that $s = f \mu^p$ has a critical point at $\msp \in Z$.
\begin{enumerate}
\item If either $p > 2$ or $p = 2$ and $n$ is odd, then $s$ has a nondegenerate critical point if and only if, in a suitable choice of $x_1,\dots,x_n$, $f_1 = 0$ and
\[
f_2 =
\begin{cases}
x_1 x_2 + x_3 x_4 + \cdots + x_{n-1} x_n, & \text{if $n$ is even}, \\
x_1^2 + x_2 x_3 + \cdots + x_{n-1} x_n, & \text{if $n$ is odd}.
\end{cases}
\]
\item If $p = 2$ and $n$ is odd, then $s$ has an almost nondegenerate critical point at $\msp$ if and only if, in a suitable choice of $x_1,\dots,x_n$, $f_1 = 0$,
\[
f_2 = \alpha x_1^2 + x_2 x_3 + \cdots + x_{n-1} x_n,
\]
for some constant $\alpha$ and the coefficient of $x_1^3$ in $f_3$ is non-zero.
\end{enumerate}
\end{Rem}

The following result is important.

\begin{Prop}[{\cite[Proposition 4.1]{Okada}, cf.\ \cite{CTP2}, \cite[Proposition 7.8]{CL}}] \label{prop:existresol}
Let the notation and assumption as above.
If $s$ has only $($almost$)$ nondegenerate critical points on $Z$, then there exists a universally $\CH_0$-trivial resolution $\varphi \colon Y \to X$ of singularities such that $\varphi^*\mcM \inj \Omega_Y^{n-1}$.
\end{Prop}

\subsection{Outline of the proof of Main Theorems} \label{sec:outline}

We explain an outline of the proof of main theorems.

We first explain an outline for del Pezzo fibrations of degree $1$ in detail.
We consider a complete linear system $|\mcO_X (6 \mu,6)|$ on a $\mbP (1,1,2,3)$-bundle $P$ over $\mbP^{n-2}$ defined by
\[
\begin{pmatrix}
u_0 & \cdots & u_{n-2} & & x & y & z & w \\
1 & \cdots & 1 & | & 0 & \lambda & 2 \mu & 3 \mu \\
0 & \cdots & 0 & | & 1 & 1 & 2 & 3
\end{pmatrix},
\]
where $n \ge 3$, $\lambda, \mu$ are integers.
Over $\mbC$, if $X \in |\mcO_P (6 \mu,6)|$ is a general member, then $\pi \colon X \to \mbP^{n-2}$, the restriction of the projection $P \to \mbP^{n-2}$, is a nonsingular del Pezzo fibration of degree $1$ under some numerical conditions on $\lambda,\mu$ (which will be considered later on).
We consider a member $X \in |\mcO_P (6 \mu,6)|$ defined by an equation
\[
w^2 + f (u,x,y,z) = 0,
\]
where $f (u,x,y,z)$ is a general polynomial in variables of bi-degree $(6 \mu,6)$.
Let $Z$ be the WPS bundle over $\mbP^{n-2}$ defined by
\[
\begin{pmatrix}
u_0 & \cdots & u_{n-2} & & x & y & z \\
1 & \cdots & 1 & | & 0 & \lambda & 2 \mu \\
0 & \cdots & 0 & | & 1 & 1 & 2
\end{pmatrix}.
\]
Then the natural morphism $\tau \colon X \to Z$ is a double cover branched along the divisor $(f = 0) \subset Z$.
We see that the nonsingular locus of $Z$ is the set $Z^{\circ} = Z \setminus (x = y = 0)$.

Now we assume that the ground field is an algebraically closed field $\K$ of characteristic $2$.
Set $\mcL = \mcO_Z (3 \mu,3)$ and $\mcL^{\circ} = \mcL|_{Z^{\circ}}$.
Then we can apply the techniques of Koll\'ar summarized in Section \ref{sec:cyccov} for $X^{\circ} \to Z^{\circ}$, where $X^{\circ} = \tau^{-1} (Z^{\circ})$, and there exists a line bundle $\mcM^{\circ} \inj (\Omega_{X^{\circ}}^{n-1})^{\vee \vee}$ such that $\mcM^{\circ} \cong \tau^* (\omega_{Z^{\circ}} \otimes {\mcL^{\circ}}^2)$.
Let $\mcM \inj (\Omega_X^{n-2})^{\vee \vee}$ be the pushforward of $\mcM^{\circ}$ via the inclusion $X^{\circ} \inj X$. 
We prove the following.

\begin{enumerate}
\item $X$ is nonsingular along $X \setminus X^{\circ}$. 
\item There exists a universally $\CH_0$-trivial resolution $\varphi \colon Y \to X$ of singularities of $X$ such that $\varphi^*\mcM \inj \Omega_Y^{n-1}$.
\item Under some conditions (e.g.\ $n = 3$ or $-K_X$ is not ample), $H^0 (X,\mcM) \ne 0$.
\end{enumerate}

Note that (1) in particular implies that $\mcM$ is an invertible sheaf.
By (2), (3) and Lemma \ref{lem:totaro}, $Y$ is not universally $\CH_0$-trivial.
Then Theorems \ref{mainthm1}, \ref{mainthm2} (for degree $1$ case) and their stronger version Theorem \ref{thm:dP1} will follow from the specialization Theorem \ref{thm:spCH0} of universal $\CH_0$-triviality.

Outlines of the proofs for degree $2$ and $3$ cases are similar and we give brief explanations.
For del Pezzo fibrations of degree $2$ or $3$, we consider a complete linear system $|\mcO_P (\delta,d)|$ on a WPS bundle over $\mbP^{n-2}$ defined by
\[
\begin{pmatrix}
u_0 & \cdots & u_{n-2} & & x & y & z & w \\
1 & \cdots & 1 & | & 0 & \lambda & \mu & \nu \\
0 & \cdots & 0 & | & 1 & 1 & 1 & m
\end{pmatrix},
\]
where $n \ge 3$, $d, \delta, \lambda,\mu,\nu$ and $m$ are integers such that $(d, m) = (4,2)$ and $(3,1)$.
Over $\mbC$, if $X \in |\mcO_X (\delta,d)|$ is a general member and $m = 2$ (resp.\ $m = 1$), then $\pi \colon X \to \mbP^{n-2}$ is a nonsingular del Pezzo fibration of degree $2$ (resp.\ $3$).
Over an algebraically closed field $\K$ of characteristic $p \in \{2,3\}$, we consider a member $X \in |\mcO_P (\delta,d)|$ admitting a morphism $\pi \colon X \to Z$ which is a branched covering of degree $p$ over a normal projective variety $Z$.
As in the degree $1$ case, let $Z^{\circ}$ be the nonsingular locus of $Z$ and $X^{\circ} = \tau^{-1} (Z^{\circ})$.
Then, corresponding to the covering $X^{\circ} \to Z^{\circ}$, there is an associated line bundle $\mcM^{\circ} \subset (\Omega_{X^{\circ}}^{n-1})^{\vee \vee}$.
Then the main part of the rest of this paper is to prove (1), (2) and (3) above, which will complete the proofs of Theorems \ref{mainthm1}, \ref{mainthm2} and their stronger versions Theorems \ref{thm:dP2}, \ref{thm:dP3}.

The proofs of (1) and (3) are straightforward, hence the proof of (2) is the central part of this paper.
In most of the cases, we can prove that the branched divisor of the covering $\tau \colon X \to Z$ has only (almost) nondegenerate critical points following the arguments similar to \cite[V.5]{Kollar}, which proves (2) by Proposition \ref{prop:existresol}.
Note that the above singularities are isolated.
However, we encounter the case when the branched divisor has critical points along a positive dimensional subvariety of $Z$, which are evidently not (almost) nondegenerate.
In the next section, we devote ourselves to give some preliminary results in order to overcome this difficulty.

\section{Some results on singularities and critical points}

\subsection{Lifting of differential forms}

Let $\msp \in Z$ be a germ of a nonsingular variety of dimension $n \ge 4$ defined over an algebraically closed field $\K$ of characteristic $p \in \{2,3\}$, $\mcL$ an invertible sheaf on $Z$ and $f \in H^0 (Z, \mcL^p)$.
We denote by $\Crit (f) \subset Z$ the set of critical points of $f$.
Let $\tau \colon X \to Z$ be the degree $p$ cover branched along $(f = 0) \subset Z$.

We consider the following conditions on $f$.

\begin{Cond} \label{cd:cr1}
There exist local coordinates $x_1,\dots,x_n$ of $Z$ with the origin at $\msp$ satisfying the following properties:
\begin{enumerate}
\item $f = \alpha + \beta x_1^2 + x_2 x_3 + \gamma x_1^3 + g$, where $\alpha,\beta,\gamma \in \K$, $g = g (x_1,\dots,x_n)$ is contained in the ideal $(x_1,x_2,x_3)^3$ and the coefficient of $x_1^3$ in $g$ is zero.
Moreover if $p = 2$, then $\gamma \ne 0$, and if $p = 3$, then $\beta \ne 0$.
\item $\Crit (f) = (x_1 = x_2 = x_3 = 0) \subset Z$.
\end{enumerate}
\end{Cond}

\begin{Cond} \label{cd:cr2}
There exist local coordinates $x_1,\dots,x_n$ of $Z$ with the origin at $\msp$ satisfying the following properties:
\begin{enumerate}
\item $f = \alpha + x_1 x_2 + x_3 x_4 + g$, where $\alpha \in \K$ and $g = g (x_1,\dots,x_n) \in (x_1,x_3)^2$.
\item $\Crit (f) = (x_1 = x_2 = x_3 = x_4 = 0) \subset Z$.
\end{enumerate}
\end{Cond}

Let $\mcM = \tau^* \mcQ \subset (\Omega_X^{n-1})^{\vee \vee}$ be the invertible sheaf associated with $\tau$ and let $\eta$ be a (rational) $(n-1)$-form which is a generator of $\mcM$.
Let $\tau^{-1} (\Crit (f))$ be the inverse image with the reduced induced scheme structure.
If $f$ satisfies one of the above conditions, then $\tau^{-1} (\Crit (f))$ is a nonsingular subvariety of $X$.
Let $\psi \colon Y \to X$ be the blowup of $X$ along $\tau^{-1} (\Crit (f))$ and $E$ its exceptional divisor.

\begin{Lem} \label{lem:liftM}
Suppose that $f$ satisfies either \emph{Condition \ref{cd:cr1}} or \emph{\ref{cd:cr2}}.
Then $\varphi \colon Y \to X$ is a resolution of singularities of $X$ and $\varphi^*\mcM \inj \Omega_Y^{n-1}$.
\end{Lem}

\begin{proof}
We need to check that $Y$ is nonsingular along $\varphi^{-1} (\msp)$.
This is straightforward and we omit the proof.
We prove the latter part assuming Condition \ref{cd:cr1}.
By \cite[Lemma V.5.9]{Kollar}, a generator $\eta$ of the invertible sheaf $\mcM$ can be expressed as
\[
\eta = \frac{d x_2 \wedge \cdots \wedge d x_n}{\prt f/\prt x_2}
= \frac{d x_2 \wedge \cdots \wedge d x_n}{x_3 + h},
\]
where $h = h (x_1,\dots,x_n) \in (x_1,\dots,x_n)^2$.
On the $x_2$-chart of $Y$, that is, the chart with coordinates $y_2  = x_2$, $y_i = x_i/y_2$ for $i \ne 2$, we have
\[
\varphi^* \eta = \frac{d y_2 \wedge d y_2 y_3 \wedge d y_4  \cdots \wedge d y_n}{y_2 y_3 + y_2^2 (\cdots)}
= \frac{d y_2 \wedge \cdots \wedge d y_n}{y_3 + y_2 (\cdots)}.
\]
Thus $\varphi^*\eta$ does not have a pole along $E$, which implies $\varphi^*\mcM \inj \Omega_Y^{n-1}$.
The proof can be done similarly when $f$ satisfies Condition \ref{cd:cr2} and we omit it.
\end{proof}

\subsection{Resolution of singularities}

Let $\K$ be an algebraically closed field of characteristic $p \in \{2,3\}$ and let $Z = \mbA^{m+2}$ be the affine space over $\K$ with affine coordinates $u_1,\dots,u_m$ and $x, y$.
The aim of this subsection is to construct a resolution of singularities of cyclic covers of $Z$ branched along $(f = 0) \subset Z$ for suitable polynomials $f = f (u,x,y)$.

\begin{Def}
For polynomials $f_1, f_2 \in \K [u,x,y]$, we write $f_1 \sim_p f_2$ if $f_1 - f_2 = h^p$ for some $h \in \K [u,x,y]$.
\end{Def}

It is easy to see that if $f_1 \sim_p f_2$, then the sets of critical points of $f_1$ and $f_2$ coincide.
We introduce the following conditions on $f$.

\begin{Cond} \label{cond:resolposcrit1}
In case $p = 2$ we have
\[
f \sim_2 a x + b x^2 + c x y + y^3 + g,
\]
and in case $p = 3$ we have
\[
f \sim_3 a x + b x^2 + c x y + y^2 + g,
\]
where
\begin{enumerate}
\item $a,b,c$ are polynomials in variables $u_1,\dots,u_m$ with $\deg (a) > 0$,
\item the hypersurface in $\mbA^m_{u_1,\dots,u_m}$ defined by $a = 0$ is nonsingular,
\item $g = g (u,x,y)$ is contained in the ideal $(x,y)^3 \subset \K [u,x,y]$,
\item if $p = 2$, then any monomial in $g$ divisible by $y^3$ is divisible by either $y^3 x$ or $y^4$, and
\item along an open subset of $Z$ containing $\Xi = (x = y = a = 0)$, the set of critical points of $f$ coincides with $\Xi$.
\end{enumerate}
\end{Cond}

\begin{Cond} \label{cond:resolposcrit2}
The characteristic $p$ of the base field is $2$ or $3$ and
\[
f \sim_p a x + b y + g,
\]
where
\begin{enumerate}
\item $a, b$ are polynomials in variables $u_1,\dots,u_m$ with $\deg (a), \deg (b) > 0$,
\item the complete intersection in $\mbA^m_{u_1,\dots,u_m}$ defined by $a = b = 0$ is nonsingular,
\item $g$ is contained in the ideal $(x,y)^2 \subset \K [u,x,y]$, and
\item along an open subset of $Z$ containing $\Xi = (x = y = a = b = 0)$, the set of critical points of $f$ coincides with $\Xi$.
\end{enumerate}
\end{Cond}

Let $X$ be the degree $p$ cover of $Z$ branched along $(f = 0) \subset Z$.
Explicitly, $X$ is the hypersurface in $\mbA^{m+3} = Z \times \mbA^1_w$ defined by $w^p + f = 0$.
We denote by $\psi \colon \tilde{X} \to X$ the blowup of $X$ along the nonsingular subvariety $\tau^{-1} (\Xi)$ and by $E$ its exceptional divisor.
Clearly $X$ is isomorphic to a hypersurface in $\mbA^{m+3}$ defined by $w^p + f_1 = 0$ for any $f_1 = f_1 (u,x,y)$ with $f \sim_p f_1$.

First we assume $f$ satisfies Condition \ref{cond:resolposcrit1}.
Our aim is to observe the blowup $\psi$, and thus we may assume 
\[
f = 
\begin{cases}
a x + b x^2 + c x y + y^3 + g, & \text{in case $p = 2$}, \\
a x + b x^2 + c x y + y^2 + g, & \text{in case $p = 3$},
\end{cases}
\]
where $a,b,c$ and $g$ are as in Condition \ref{cond:resolposcrit1}, and that $X$ is the hypersurface defined by
\[
F := w^p + f = 0.
\]
Note that $\tau^{-1} (\Xi) = (w = x = y = a = 0)$ and we this subvariety as $\Sigma$.

\begin{Lem} \label{lem:resol1}
Suppose that $f$ satisfies \emph{Condition \ref{cond:resolposcrit1}}. 
Then the variety $\tilde{X}$ is nonsingular on an open subset containing $E$.
Moreover, there is an isomorphism
\[
E \cong S \times \Sigma,
\]
where 
\[
S = (\delta_{p,2} \tilde{w}^2 + \tilde{t} \tilde{x} + \delta_{p,3} \tilde{y}^2 = 0) \subset \mbP^3_{\tilde{t},\tilde{x},\tilde{y},\tilde{w}},
\]
and $\psi|_E \colon E \to S$ coincides with the projection $S \times \Sigma \to \Sigma$ $($Here $\delta_{p,i}$ is the Kronecker delta$)$.
\end{Lem}

\begin{proof}
The smoothness of $\tilde{X}$ along $E$ can be checked \'etale locally on the base $X$ and thus follows from Lemma \ref{lem:liftM}.

We give a proof assuming $p = 2$.
The proof is similar when $p = 3$.
In order to visualize the blowup $\psi$, we introduce a new variable $t$ and let $U = \mbA^{m+4}$ be the affine space with coordinates $u_1,\dots,u_m,t,x,y,w$.
Then $X$ is naturally isomorphic to the complete intersection in $U$ defined by
\[
w^2 + t x + b x y + c x^2 + y^3 + g = t - a = 0.
\]
Replacing $t \mapsto t - b y - c x$, the above equation can be written as
\[
w^2 + t x + y^3 + g = t - a - h = 0,
\]
where $g, h \in \K [u,x,y]$ with $g \in (x,y)^3$ and $h \in (x,y)$.
Let $\Psi \colon V \to U$ be the blowup along $(t = x = y = w = 0)$.
Then $\psi \colon \tilde{X} \to X$ is identified with the restriction of $\Psi$ to the proper transform of $X \subset U$ via $\Psi$.
The variety $V$ is covered by standard affine open charts $V_t, V_x, V_y$ and $V_w$, which will be described below.

The $t$-chart $V_t$ is an affine space $\mbA^{m+4}$ with coordinates $u_1,\dots,u_m$, $\tilde{t} = t$, $\tilde{x} = x/t$, $\tilde{y} = y/t$ and $\tilde{w} = w/t$.
$\tilde{X} \cap V_t$ is defined by
\[
\tilde{w}^2 + \tilde{x} + \tilde{t} \tilde{y}^3 + \tilde{t} \tilde{g}_t = \tilde{t} - a - \tilde{t} \tilde{h}_t = 0,
\]
where $\tilde{g}_t = g (u,\tilde{t} \tilde{x},\tilde{t} \tilde{y})/\tilde{t}^3$ and $\tilde{h}_t = h (u,\tilde{t} \tilde{x},\tilde{t} \tilde{y})/\tilde{t}$.
On this chart, the exceptional divisor $E$ is cut out by $\tilde{t} = 0$ and we have
\[
E \cap V_t = (\tilde{t} = \tilde{w}^2 + \tilde{x} = a = 0) \subset V_t.
\]

The $x$-chart $V_x$ is an affine space $\mbA^{m+4}$ with coordinates $u_1,\dots,u_m$, $\tilde{t} = t/x$, $\tilde{x} = x$, $\tilde{y} = y/x$ and $\tilde{w} = w/x$.
$\tilde{X} \cap V_x$ is defined by
\[
\tilde{w}^2 + \tilde{t} + \tilde{x} \tilde{y}^3 + \tilde{x} \tilde{g}_x = \tilde{x} \tilde{t} - a - \tilde{x} \tilde{h}_x = 0,
\]
where $\tilde{g}_x = g (u,\tilde{x},\tilde{x} \tilde{y})/\tilde{x}^3$ and $\tilde{h}_x = h (u,\tilde{x},\tilde{x} \tilde{y})/\tilde{x}$, and we have
\[
E \cap V_x = (\tilde{x} = \tilde{w}^2 + \tilde{t} = a = 0) \subset V_x.
\]

The $y$-chart $V_y$ is an affine space $\mbA^{m+4}$ with coordinates $u_1,\dots,u_m$, $\tilde{t} = t/y$, $\tilde{x} = x/y$, $\tilde{y} = y$ and $\tilde{w} = w/y$.
$\tilde{X} \cap V_y$ is defined by
\[
\tilde{w}^2 + \tilde{t} \tilde{x} + \tilde{y} + \tilde{y} \tilde{g}_y = \tilde{y} \tilde{t} - a - \tilde{y} \tilde{h}_y = 0,
\]
where $\tilde{g}_y = g (u,\tilde{y} \tilde{x},\tilde{y})/\tilde{y}^3$ and $\tilde{h}_y = h (u, \tilde{y} \tilde{x}, \tilde{y})/\tilde{y}$, and we have
\[
E \cap V_y = (\tilde{y} = \tilde{w}^2 + \tilde{t} \tilde{x} = a = 0) \subset V_y.
\]

The $w$-chart $V_w$ is an affine space $\mbA^{m+4}$ with coordinates $u_1,\dots,u_m$, $\tilde{t} = t/w$, $\tilde{x} = x/w$, $\tilde{y} = y/w$ and $\tilde{w} = w$.
$\tilde{X} \cap V_w$ is defined by
\[
1 + \tilde{t} \tilde{x} + \tilde{w} \tilde{y}^3 + \tilde{w} \tilde{g}_w = \tilde{w} \tilde{t} - a - \tilde{w} \tilde{h}_w = 0,
\]
where $\tilde{g}_w = g (u, \tilde{w} \tilde{x},\tilde{w} \tilde{y})/\tilde{w}^3$ and $\tilde{h}_x = h (u, \tilde{w} \tilde{x},\tilde{w} \tilde{y})/\tilde{w}$, and we have
\[
E \cap V_w = (\tilde{w} = 1 + \tilde{t} \tilde{x} = a = 0) \subset V_w.
\]

By gluing $E \cap V_t,\dots,E \cap V_w$, we see that $E$ is isomorphic to 
\[
(\tilde{w}^2 + \tilde{t} \tilde{x} = a = 0) \subset \mbP^3 \times \mbA^m_u,
\]
which coincides with $S \times \Sigma$, and the morphism $\psi|_E \colon E \to \Sigma$ is the projection $S \times \Sigma \to \Sigma$.
\end{proof}

Next, we assume that $f$ satisfies Condition \ref{cond:resolposcrit2}.
Then we may assume $f = a x + b y + g$, where $a,b$ and $g$ are as in Condition \ref{cond:resolposcrit2}, and that $X$ is the hypersurface in $\mbA^{m+3}$ defined by the equation
\[
F := w^p + a x + b y + g = 0.
\]
Note that $\tau^{-1} (\Xi) = (w = x = y = a = b = 0)$.
\begin{Lem} \label{lem:resol2}
The variety $\tilde{X}$ is nonsingular on an open subset containing $E$.
Moreover, there is an isomorphism
\[
E \cong S \times \Sigma,
\]
where 
\[
S = (\delta_{p,2} \tilde{w}^2 + \tilde{s} \tilde{x} + \tilde{t} \tilde{y} = 0) \subset \mbP^4_{\tilde{s}, \tilde{t},\tilde{x},\tilde{y},\tilde{w}},
\]
and $\psi|_E \colon E \to S$ coincides with the projection $S \times \Sigma \to \Sigma$ $($Here, $\delta_{p,2}$ is the Kronecker delta$)$.
\end{Lem}

\begin{proof}
The proof is similar to that of Lemma \ref{lem:resol1}.
The smoothness of $\tilde{X}$ follows from Lemma \ref{lem:liftM}.
Let $U = \mbA^{m+5}$ be the affine space with coordinates $u,s,t,x,y,w$.
The variety $X$ is isomorphic to the complete intersection in $U$ defined by
\[
w^p + s x + t y + g = s - a = t - b = 0.
\]
Filtering off terms divisible by $x, y$ and then replacing $s, t$, the above equations can be written as
\[
w^p + s x + t y = s - a - q = t - b - h = 0,
\]
where $q = q (u,x,y), h = h (u,x,y)$ are contained in $(x,y)$.
Let $\Psi \colon V \to U$ be the blowup of $U$ along $(s = t = x = y = w = 0)$.
Then $\psi \colon \tilde{X} \to X$ can be identified with the restriction of $\Psi$ to the proper transform of $X \subset U$.
The variety $V$ is covered by standard affine charts $V_s, V_t,V_x,V_y$ and $V_w$.
The description of these charts are similar to those in the proof of Lemma \ref{lem:resol1}, and the description of $E$ and $\psi|_E \colon E \to \Sigma$ follows.
\end{proof}

\begin{Rem} \label{rem:CHtriv}
Let $\psi \colon \tilde{X} \to X$ be as in Lemma \ref{lem:resol1} or \ref{lem:resol2}.
Let $\msp \in X$ be a scheme point and we consider the fiber $\psi^{-1} (\msp)$, viewed as a scheme over the residue field $k (\msp)$.
If $\msp \notin \Sigma$, then $\varphi^{-1} (\msp)$ is a point.
If $\msp \in \Sigma$, then $\psi^{-1} (\msp)$ is $S_{k (\msp)}$ defined over $k (\msp)$, where $S$ is the quadric hypersurface given in Lemma \ref{lem:resol1} or \ref{lem:resol2}.
Clearly $S_{k (\msp)}$ is universally $\CH_0$-trivial.
Therefore, by \cite[Proposition 1.8]{CTP}, the morphism $\psi$ is universally $\CH_0$-trivial. 
\end{Rem}

\subsection{Smoothness of certain hypersurfaces}

\begin{Lem} \label{lem:crismhyp}
Let $\mbA^{n+1}$ be an affine space with coordinates $x_1,\dots,x_n, w$, and let $X \subset \mbA^{n+1}$ be the hypersurface defined by 
\[
F := w g (x_1,\dots,x_n) + f (x_1,\dots,x_n) = 0.
\]
Suppose that the complete intersection in $\mbA^n$ with coordinates $x_1,\dots,x_n$ defined by $f = g =0$ is nonsingular.
Then $X$ is nonsingular.
\end{Lem}

\begin{proof}
It is clear that the singular locus $\Sing (X)$ is contained in the closed subset $(g = 0) \subset X$.
We set $V = (f = g = 0) \subset \mbA^n$ and let $J_V$ be the Jacobian matrix of $V \subset \mbA^n$.
Suppose that $X$ has a singular point $\msp = (\alpha_1,\dots,\alpha_n,\beta)  \in X$ and set $\msq = (\alpha_1,\dots,\alpha_n) \in \mbA^n$.
Then $\msq \in V$ and we have
\[
\frac{\prt g}{\prt x_i} (\msq) \beta + \frac{\prt f}{\prt x_i} (\msq) = 0
\]
for $i = 1,\dots,n$.
This implies
\[
({}^t \! J_V) (\msq) \begin{pmatrix}
\beta \\
1
\end{pmatrix} = 
\begin{pmatrix}
0 \\
\vdots \\
0
\end{pmatrix},
\]
which is impossible since the rank of $J_V (\msq)$ is $2$.
Therefore $X$ is nonsingular.
\end{proof}

\subsection{Surjectivity of restriction maps}

We study surjectivity of restriction maps of global sections of a line bundle.
We work over an algebraically closed field.

\begin{Def}
Let $Z$ be a normal quasi-projective variety and $\mcN$ an invertible sheaf on $Z$.
For a positive integer $k$ and a nonsingular point $\msp \in Z$, the restriction map
\[
r_{\mcN,k} (\msp): H^0 (Z, \mcN) \to \mcN \otimes (\mcO_Z/\mfm^k_{\msp}),
\]
where $\mfm_{\msp}$ is the maximal ideal of $\mcO_{Z,\msp}$, is called the {\it $k$th restriction map of $\mcN$ at $\msp$}.
\end{Def}

Let $Q$ be a WPS bundle over $\mbP^n$ defined by
\[
\begin{pmatrix}
u_0 & \cdots & u_n & & x & y & z \\
1 & \cdots & 1 & | & 0 & \lambda & \mu \\
0 & \cdots & 0 & | & 1 & 1 & m
\end{pmatrix},
\] 
where $m$ is a positive integer.
We assume that $\lambda, \mu \ge 0$.
We define 
\[
U_x = (x \ne 0) \subset Q, \ 
\Pi_y = (x = 0) \cap (y \ne 0) \subset Q, \  
\Gamma_z = (x = y = 0) \subset Q,
\]
so that $Q$ is the disjoin union of $U_x$, $\Pi_y$ and $\Gamma_z$.

For positive integers $\delta$, $d$ and $k$, we consider the restriction map
\[
r_k (\msp) \colon H^0 (Q, \mcO_Q (\delta,d)) \to \mcO_Q (\delta,d) \otimes (\mcO_Q/\mfm_{\mfp}^k),
\]
and consider its surjectivity.

\begin{Lem} \label{lem:surjrest}
Suppose that $d \ge 3 m$ and $0 \le \lambda, \mu$.
Then the following hold.
\begin{enumerate}
\item If $\delta \ge \max \{2, 2 \lambda, 2 \mu\}$, then $r_3 (\msp)$ is surjective for any $\msp \in U_x$.
\item If $\delta \ge \max \{3, 3 \lambda, 3 \mu\}$, then $r_4 (\msp)$ is surjective for any point $\msp \in U_x$.
\item If $\mu \ge m \lambda$ and $\delta \ge \max \{ d \lambda + 1, (d-m) \lambda + \mu \}$, then $r_2 (\msp)$ is surjective for any point $\msp \in \Pi_y$.
\item If $m = 1$, $\lambda \le \mu$ and $\delta \ge d \mu + 1$, then $r_2 (\msp)$ is surjective for any $\msp \in \Gamma_z$.
\end{enumerate}
\end{Lem}

\begin{proof}
Set $\mcL = \mcO_Q (\delta,d)$.
A global section of $\mcL$ is a linear combination of the monomials
\[
\{ u_0^{i_1} \cdots u_n^{i_n} x^j y^k z^l \mid i_0 + \cdots + i_n + k \lambda + l \mu = \delta, \ 
j + k + m l = d \}.
\]
We prove (1) and (2).
By replacing coordinates we may assume $\msp = (1\!:\!0\!:\!\cdots\!:\!0 ; 1\!:\!0\!:\!0)$.
If $\delta \ge \max \{2, 2\lambda, 2 \mu\}$, then
\[
u_i u_j u_0^{\delta - 2} x^d, \ 
u_i u_0^{\delta - \lambda -1} x^{d-1} y, \ 
u_i u_0^{\delta - \mu - 1} x^{d-m} z,
\]
\[
u_0^{\delta - 2 \lambda} x^{d-m} y^2, \ 
u_0^{\delta - \lambda - \mu} x^{d-m-1} y z, \
u_0^{\delta - 2 \mu} x^{d-2m} z^2,
\]
where $0 \le i,j \le n$, are contained in $H^0 (Q,\mcL)$ and they restricts to basis of $\mcL \otimes (\mcO_Z/\mfm^3_{\msp})$.
Similarly, if $\delta \ge \max \{3, 3 \lambda, 3\mu\}$, then 
\[
u_i u_j u_k u_0^{\delta - 3} x^d, \ 
u_i u_j u_0^{\delta - \lambda -2} x^{d-1} y, \ 
u_i u_j u_0^{\delta - \mu -2} x^{d-m} z,
\]
\[
u_i u_0^{\delta - 2 \lambda -1} x^{d-2} y^2, \ 
u_i u_0^{\delta - \lambda - \mu -1} x^{d-1-m} y z, \ 
u_i u_0^{\delta - 2 \mu -1} x^{d-2m} z^2,
\]
\[
u_0^{\delta - 3 \lambda} x^{d-3} y^3, \ 
u_0^{\delta - 2 \lambda - \mu} x^{d-2-m} y^2 z, \ 
u_0^{\delta - \lambda - 2 \mu} x^{d-1-2m} y z^2, \ 
u_0^{\delta - 3 \mu} x^{d - 3m} z^3,
\]
where $0 \le i,j,k \le n$, are contained in $H^0 (Q,\mcL)$ and they restrict to a basis of $\mcL \otimes (\mcO_Q/\mfm_{\msp}^4)$.
This proves (1) and (2).

We prove (3).
Let $\msp \in \Pi_y$ be a point.
By the assumption $\mu \ge m \lambda$, we may assume $\msp = (1\!:\!0\!:\!\cdots\!:\!0 ; 0\!:\!1\!:\!0)$.
Then, since $\delta \ge \max \{d \lambda + 1, (d-m) \lambda + \mu \}$, we see that
\[
u_i u_0^{\delta - d \lambda -1} y^d, \ 
u_0^{\delta - (d-1) \lambda} x y^{d-1}, \ 
u_0^{\delta - (d-m) \lambda - \mu} y^{d-m} z,
\]
where $0 \le i \le n$, are contained in $H^0 (Q,\mcL)$ and they restrict to a basis of $\mcL \otimes (\mcO_Q/\mfm_{\msp}^2)$. 
Thus (3) is proved.

Finally we prove (4).
By (3), $r_2 (\msp)$ is surjective at any point $\msp \in \Pi_y$ and it remains to prove the surjectivity along $\Gamma_z$.
We may assume $\msp =(1\!:\!0\!:\!\cdots\!:\!0 ; 0\!:\!0\!:\!1)$.
Then, since $\delta \ge d \mu + 1 = \max \{d \mu + 1, (d-1) \mu + \lambda \}$, we see that
\[
u_i u_0^{\delta -d \mu + 1} z^d, \ 
u_0^{\delta - (d-1) \mu} x z^{d-1}, \ 
u_0^{\delta -(d-1) \mu - \lambda} y z^{d-1},
\]
where $0 \le i \le n$, are contained in $H^0 (Q,\mcL)$ and they restrict to basis of $\mcL \otimes (\mcO_Q/\mfm_{\msp}^2)$.
This completes the proof.
\end{proof}

%

\begin{Rem}
Suppose that the characteristic of the ground field is $p > 0$ and let $\mcL$ be an invertible sheaf on $Q$.

The surjectivity of the 3rd (resp.\ 4th) restriction map of $\mcL^p$ on $U_x$ implies that a general section $f \in H^0 (Q,\mcL^p)$ has only nondegenerate (resp.\ almost nondegenerate) critical points on $U_x$ when $p \ne 2$ or $p = 2$ and $\dim Q$ is even (resp.\ $p = 2$ and $\dim Q$ is odd).  

The surjectivity of the 2nd restriction map of $\mcL^p$ on the set $\Xi$, where $\Xi = \Pi_y, \Gamma_z$ or $\Pi_y \cup \Gamma_z$, implies that a general section $f \in H^0 (Q, \mcL^p)$ does not have a critical point on $\Xi$. 
This is because the surjectivity imposes $\dim Q$ independent conditions for a global section of $\mcL^p$ to have a critical point at a given point $\msp \in \Xi$ and on the other hand we have $\dim \Xi < \dim Q$.
\end{Rem}

\section{Del Pezzo fibrations of degree $1$}


For integers $n, \lambda$ and $\mu$ such that $n \ge 3$ and $\lambda \ge 0$, we denote by $P_1 (n, \lambda,\mu)$ the $\mbP (1,1,2,3)$-bundle over $\mbP^{n-2}$ defined by
\[
\begin{pmatrix}
u_0 & \cdots & u_{n-2} & & x & y & z & w \\
1 & \cdots & 1 &  | & 0 & \lambda & 2 \mu & 3 \mu \\
0 & \cdots & 0 &  | & 1 & 1 & 2 & 3
\end{pmatrix},
\]
and consider the complete linear system $|\mcO_P (6 \mu,6)|$, where $P = P_1 (n,\lambda,\mu)$.
The aim of the present section is to prove the following.

\begin{Thm} \label{thm:dP1}
Suppose that the ground field is $\mbC$.
Let $X$ be a very general member of $|\mcO_P (6 \mu,6)|$, and suppose that $\pi \colon X \to \mbP^{n-2}$ is a nonsingular del Pezzo fibration.
If $4 \mu - \lambda - (n-1) \ge 0$, then $X$ is not stably rational.
\end{Thm}

\begin{Cor} \label{cor:dP1}
Let $X$ be as in \emph{Theorem \ref{thm:dP1}} $($without assuming $4 \mu - \lambda - (n-1) \ge 0$$)$.
If either $n \in \{3,4\}$ or $-K_X$ is not ample, then $X$ is not stably rational.
\end{Cor}

\begin{Lem} \label{lem:dP1num}
Suppose that the ground field is $\mbC$.
Let $X$ be a general member of $|\mcO_P (6 \mu,6)|$ and suppose that $\pi \colon X \to \mbP^{n-2}$ is a nonsingular del Pezzo fibration.
Then the following hold.
\begin{enumerate}
\item $\mu > 0$.
\item If $\mu < \lambda$, then $6 \mu = 5 \lambda$.
\end{enumerate}
\end{Lem}

\begin{proof}
Let $F = F (u,x,y,z,w)$ be a defining polynomial of $X$.

We prove (1).
Suppose that $\mu < 0$.
Then $F$ is a polynomial in variables $z,w$ and $X$ is clearly singular.
This is a contradiction and we have $\mu \ge 0$.
Suppose that $\mu = 0$.
If $\lambda > 0$, then $F$ does not involve the variable $y$.
This implies that $X$ is singular along $(x = z = w = 0) \subset X$.
Thus $\lambda = 0$ and $P = P_1 (n,0,0)$ is the direct product $\mbP^{n-2} \times \mbP (1,1,2,3)$ and $X$ is isomorphic to $\mbP^{n-2} \times S$, where $S$ is a (smooth) del Pezzo surface of degree $1$.
This implies $\rho (X) > 2$ and $X$ is not a Mori fiber space.
This is a contradiction and (1) is proved. 

We prove (2).
Suppose that $\mu < \lambda$.
Then we can write
\[
F = x g (u_0,u_1,x,y,z,w) + w^2 + z^3,
\]
for some polynomial $g$ of bi-degree $(6 \mu,5)$.
If $5 \lambda > 6 \mu$, then $g$ does not contain a monomial divisible by $y^5$, which implies that $X$ is singular along $(x = z = w = 0)$.
If $5 \lambda < 6 \mu$, then the terms in $g$ divisible by $y^5$ can be written as $a (u_0,u_1) y^5$, where $a (u_0,u_1)$ is homogeneous of degree $6 \mu - 5 \lambda > 0$.
But then $X$ is singular at any point of $(a = x = z = w = 0) \subset P$ which is non-empty.
Thus $6 \mu = 5 \lambda$ and the proof is completed.
\end{proof}

\begin{Lem} \label{lem:dP1ample}
Let the notation and assumption as in \emph{Lemma \ref{lem:dP1num}}.
\begin{enumerate}
\item If $-K_X$ is not ample, then $3 \mu \ge \lambda + n-1$.
\item If $n \in \{3,4\}$, then $4 \mu - \lambda - (n-1) \ge 0$.
\end{enumerate}
\end{Lem}

\begin{proof}
Take $F_P \in |\mcO_P (1,0)|$, $D_P \in |\mcO_P (0,1)|$, and set $F = F_P|_X$, $D = D_P|_X$.
By adjunction, we have $-K_X \sim (n-1 + \lambda - \mu) F + D$. 

We prove (1).
Suppose that $\lambda \le \mu$.
Then the complete linear system $|\mcO_P (6 \mu,6)|$ is base point free.
This implies that $\mu F + D$ is nef.
Since $\rho (X) = 2$, a divisor $\alpha F + D$ is ample if $\alpha > \mu$, and the assertion follows immediately.
Suppose that $\lambda > \mu$.
Then $|\mcO_P (6 \lambda,6)|$ is base point free.
This implies that $\lambda F + D$ is nef.
It follows that a divisor $\alpha F + D$ is ample if $\alpha > \lambda$, and we obtain the inequality $n-1 \le \mu$.
Combining this inequality and $6 \mu = 5 \lambda$, it is easy to check that the inequality $3 \mu \ge \lambda + n-1$ holds.

We prove (2).
Suppose to the contrary that $4 \mu - \lambda - (n-1) < 0$.
If $\mu \ge \lambda$, then $3 \lambda < n-1 \le 3$ and hence $\lambda = 0$.
But then $4 \mu < n-1 \le 3$.
This is a contradiction since $\mu > 0$ by Lemma \ref{lem:dP1num}.
If $\mu < \lambda$, then $6 \mu = 5 \lambda$ and we have $(7/3) \lambda < n-1 \le 3$.
This implies $\lambda \le 1$.
This is again a contradiction since $6 \mu = 5 \lambda$ and $\mu > 0$, and we obtain the desired inequality.
\end{proof}

In the rest of the present section, we assume that $(n,\lambda,\mu)$ satisfies the assumption of Lemma \ref{lem:dP1num} and we keep the following setting.

\begin{Setting} \label{setting:dP1}
\begin{itemize}
\item The ground field $\K$ is algebraically closed and $\chara (\K) = 2$.
\item $X$ is a hypersurface of bi-degree $(6 \mu,6)$ in $P = P_1 (n,\lambda,\mu)$ defined by
\[
F := w^2 + f (u,x,y,z) = 0,
\]
where $f (u,x,y,z)$ is a general polynomial of bi-degree $(6 \mu,6)$.
\item $Z$ the $\mbP (1,1,2)$-bundle over $\mbP^{n-2}$ defined by 
\[
\begin{pmatrix}
u_0 & \cdots & u_{n-2} & & x & y & z \\
1 & \cdots & 1 &  | & 0 & \lambda & 2 \mu \\
0 & \cdots & 0 &  | & 1 & 1 & 2
\end{pmatrix}.
\]
\item $\mcL := \mcO_Z (3 \mu,3)$,
\item $\tau \colon X \to Z$ is the restriction of the projection $P \ratmap Z$, which is the double cover of $Z$ branched along $f \in H^0 (Z,\mcL^2)$.
\item $Z^{\circ} := Z \setminus (x = y = 0)$ is the smooth locus of $Z$, $X^{\circ} = \tau^{-1} (Z^{\circ})$ and $\tau^{\circ} = \tau|_{X^{\circ}}$.
\item $\mcM^{\circ}$ is the invertible sheaf associated to $\tau^{\circ}$ and $\mcM = \iota_* \mcM^{\circ}$, where $\iota \colon X^{\circ} \inj X$.
\end{itemize}
\end{Setting}

We define 
\[
U_x = (x \ne 0) \subset Z \quad \text{and} \quad \Pi_y = (x = 0) \cap (y \ne 0) \subset Z,
\]
so that we have $Z^{\circ} = U_x \cup \Pi_y$.
We will show that $X$ admits a universally $\CH_0$-trivial resolution $\varphi \colon Y \to X$ of singularities such that $\varphi^*\mcM \inj \Omega_Y^{n-1}$.
To do so we need to analyze the critical points of $f \in H^0 (Z^{\circ},\mcL^2)$, which will be done in the following subsections.

\subsection{Case $\lambda \ne \mu$} \label{sec:dP1-1}

\begin{Lem} \label{lem:critdP1case1}
A general $f \in H^0 (Z,\mcL^2)$ has only $($almost$)$ nondegenerate critical points on $Z^{\circ}$ and $X$ is nonsingular along $X \setminus X^{\circ}$.
\end{Lem}

\begin{proof}
We have $6 \mu \ge \max \{3, 3 \lambda, 3 \mu\}$ by Lemma \ref{lem:dP1num}.
By Lemma \ref{lem:surjrest}, the 4th restriction map of $\mcL^2$ is surjective at any point $\msp \in U_x$.
Thus a general $f \in H^0 (Z,\mcL^2)$ has only (almost) nondegenerate critical points on $U_x$.

We consider critical points on $\Pi_y$.
We first consider the case $\mu > \lambda$.
In this case, $6 \mu \ge \max \{ 6 \lambda + 1, 4 \lambda + 2 \mu \}$.
By Lemma \ref{lem:surjrest}, the 2nd restriction map of $\mcL^2$ is surjective at any $\msp \in \Pi_y$ and hence a general $f \in H^0 (Z,\mcL^2)$ does not have a critical point on $\Pi_y$.
We next consider the case $\mu < \lambda$.
In this case $6 \mu = 5 \lambda$ and we can write
\[
f = z^3 + x (y^5 + a_{4 \mu - 3 \lambda} y^3 z + b_{2 \mu - \lambda} y z^2) + x^2 g (u,x,y,z),
\] 
where $g (u,x,y,z)$ is of bi-degree $(6 \mu, 4)$.
It is easy to see that if $f$ has a critical point at $\msp \in \Pi_y$, then
\[
\msp \in \left( x = \frac{\prt f}{\prt x} = \frac{\prt f}{\prt z} = 0 \right) \cap \Pi_y.
\]
But the above set is clearly empty and thus $f$ cannot have a critical point along $\Pi_y$.
Therefore the first assertion is proved.

We prove the latter assertion. 
We see that $X$ is defined by an equation
\[
F := w^2 + z^3 + g (u,x,y,z,w) = 0,
\]
where $g \in (x,y)$.
We have 
\[
X \setminus X^{\circ} = (x = y = 0) \cap X = (x = y = w^2 + z^3 = 0),
\]
which implies that $z$ does not vanish at any point $\msp \in X \setminus X^{\circ}$.
Thus $X$ is nonsingular along $X \setminus X^{\circ}$ since $\prt F/\prt z = z^2 \ne 0$ along $X \setminus X^{\circ}$.
\end{proof}


\subsection{Case $\lambda = \mu$} \label{sec:dP1-2}

\begin{Lem} \label{lem:critdP1case2}
A general $f \in H^0 (Z,\mcL^2)$ has only $($almost$)$ nondegenerate critical points on $U_x$ and $X$ is nonsingular along $X \setminus X^{\circ}$.
\end{Lem}

\begin{proof}
This can be proved in the same way as Lemma \ref{lem:critdP1case2} since $6 \mu \ge \max \{3, 3 \lambda, 3 \mu\}$.
\end{proof}

Let $\Crit (f) \subset Z^{\circ}$ be the set of critical points of $f$ on $Z^{\circ}$.
Lemma \ref{lem:critdP1case2} shows that $C_1 := \Crit (f) \cap U_x$ consists of (almost) nondegenerate critical points.
We need to consider $\Crit (f) \setminus C_1$, or in other words, the critical points of $f$ on $\Pi_y$.
We can write
\[
f = \alpha y^6 + \beta y^4 z + \gamma y^2 z^2 + z^3 + x (a_{\mu} y^5 + b_{\mu} y^3 z + c_{\mu} y z^2) + x^2 g, 
\]
where $\alpha, \beta, \gamma \in \K$, $a_{\mu},b_{\mu},c_{\mu} \in \K [u]$ are of degree $\mu$ and $g = g (u,x,y,z)$ is of bi-degree $(6 \mu,4)$.
Replacing $z$, we may assume that $\beta = 0$.

\begin{Lem} \label{lem:dP1critloc}
We have 
\[
\Crit (f) = C_1 \cup C_2,
\]
where $C_2 = (x = z = a_{\mu} = 0) \subset Z$.
\end{Lem}

\begin{proof}
This follows immediately since
\[
\frac{\prt f}{\prt x} |_{(x = 0)} = a_{\mu} y^5 + b_{\mu} y^3 z + c_{\mu} y z^2, \ 
\frac{\prt f}{\prt z} |_{(x = 0)} = z^2,
\]
and $\prt f/\prt u_i$ vanishes along $(x = 0)$ for any $i$.
\end{proof}

By Lemma \ref{lem:dP1critloc}, the singular locus of $X$ is $\tau^{-1} (C_1) \cup \tau^{-1} (C_2)$.
Let $\varphi \colon Y \to X$ be the composite of blowups of $X$ at each point of $\tau^{-1} (C_1)$ and along $\tau^{-1} (C_2)$.
The singularities $\tau^{-1} (C_1)$ correspond to (almost) nondegenerate critical points and $\varphi$ gives a desired resolution of such a singularities (see Proposition \ref{prop:existresol} and also \cite[Proposition 4.1]{Okada}).
Note that if $n = 3$, then $C_2$ is also isolated and $f$ has almost nondegenerate critical points at each point of $C_2$.
We will observe that $\varphi$ gives a desired resolution of singularities along $\tau^{-1} (C_2) \subset X$ assuming $n > 3$.

Recall that
\[
f = \alpha y^6 + \gamma y^2 z^2 + z^3 + x (a_{\mu} y^5 + b_{\mu} y^3 z + c_{\mu} y z^2) + x^2 g, 
\]
and
\[
C_2 = (x = z = a_{\mu} = 0).
\]
We write $g = y^4 d_{2 \mu} + h$, where $d_{2 \mu} = d_{2 \mu} (u)$ and $h = h (u,x,y,z)$ does not contain a monomials divisible by $y^4$.
Note that $h \in (x,z)$.
We see that $C_2 \subset Z$ is covered by open sets $U_{i,y} = (u_i \ne 0) \cap (y \ne 0) \subset Z$ for $i = 0,\dots,n-2$.
In the following we work with $U_{0,y}$ and analyze the restriction of $f$ to $U_{0,y}$.
By symmetry the descriptions for the other $U_{i,y}$ are completely the same.

The open set $U_{0,y}$ is isomorphic to the affine space $\mbA^n$ and the restriction of the sections $u_1,\dots,u_{n-2},x,z$ to $U_0$ are the affine coordinates of $U_0 \cong \mbA^n$.
By a slight abuse of notation, the restriction of $u_i,x,z$ to $U_0$ are also denoted by the same symbol.
For a polynomial $e = e (u_0,\dots,u_{n-2},x,y,z)$, we define
\[
\bar{e} = \bar{e} (u_1,\dots,u_{n-2},x,z) = e (1,u_1,\dots,u_{n-2},x,1,z).
\]
We have
\[
\begin{split}
f|_{U_{0,y}} &= \alpha + \gamma z^2 + z^3 + a_{\mu} x + b_{\mu} x z + c_{\mu} x z^2 + d x^2 + x^2 \bar{h} \\
& \sim_2 a_{\mu} x + d x^2 + b_{\mu} x z + z^3 + (c_{\mu} x z^2 + x^2 \bar{h})
\end{split}
\]
We see that $\deg a_{\mu} = \mu > 0$ and the hypersurface in $\mbA^{n-2}_u$ defined by $a_{\mu} = 0$ is nonsingular since $f$ is general.
Moreover $c_{\mu} x z^2 + x^2 \bar{h}$ is contained in the ideal $(x,z)^3$ and it does not contain a monomial divisible by $z^3$.
This shows that $f|_{U_{0,y}}$ satisfies Condition \ref{cond:resolposcrit1} and, by Lemma \ref{lem:resol1}, the blowup of $X$ along $\tau^{-1} (C_2)$ is universally $\CH_0$-trivial, resolves the singularity of $X$ along $\tau^{-1} (C_2)$ and pulls back the sheaf $\mcM$ into a subsheaf of $\Omega^{n-1}$.

The following is the conclusion of Sections \ref{sec:dP1-1} and \ref{sec:dP1-2}.

\begin{Prop} \label{prop:goodresoldP1}
Suppose that $(n,\lambda,\mu)$ satisfies the assumption of \emph{Lemma \ref{lem:dP1num}} and let $X$ be as in \emph{Setting \ref{setting:dP1}}.
Then there exists a universally $\CH_0$-trivial resolution $\varphi \colon Y \to X$ of singularities such that $\varphi^* \mcM \inj \Omega_Y^{n-1}$ and we have
\[
\mcM \cong \mcO_X (4 \mu - \lambda - (n-1), 2).
\]
\end{Prop}

\begin{proof}
The existence of the universally $\CH_0$-trivial resolution $\varphi \colon Y \to X$ follows from the results of Sections \ref{sec:dP1-1}, \ref{sec:dP1-2} and Proposition \ref{prop:existresol}.
We have $\omega_Z \cong \mcO_Z (-(n-1)-\lambda - 2 \mu,-4)$ and $\mcL \cong \mcO_Z (3 \mu,3)$, so that
\[
\mcM \cong \tau^*(\omega_Z \otimes \mcL^2) \cong \mcO_X (4 \mu - \lambda - (n-1),2),
\]
and the proof is completed.
\end{proof}

\subsection{Proof of Theorem \ref{thm:dP1} and Corollary \ref{cor:dP1}}

\begin{proof}[Proof of \emph{Theorem \ref{thm:dP1}}]
Let $X$ be as in Setting \ref{setting:dP1}.
By Proposition \ref{prop:goodresoldP1}, there exists a universally trivial resolution $\varphi \colon Y \to X$ such that $\varphi^*\mcM \inj \Omega_Y^{n-1}$.
The assumption $4 \mu - \lambda - (n-1) \ge 0$ implies that $H^0 (X, \mcM) \ne 0$ and hence $H^0 (Y,\Omega_Y^{n-1}) \ne 0$.
By Lemma \ref{lem:totaro}, $Y$ is not universally $\CH_0$-trivial.
Now we assume that $f \in H^0 (Z, \mcL^2)$ is very general so that the coefficients of $f$ are algebraically independent over $\mbF_2$.
We can lift it to characteristic $0$ via the ring of Witt vectors which is a DVR with residue field $\K$, and then to $\mbC$ by choosing an embedding of the fraction field of the DVR into $\mbC$.
Thus Theorem \ref{thm:dP1} follows from Theorem \ref{thm:spCH0} (see Section \ref{sec:outline} for a more detailed explanation).
\end{proof}

\begin{proof}[Proof of \emph{Corollary \ref{cor:dP1}}]
This follows from Theorem \ref{thm:dP1} and Lemma \ref{lem:dP1ample}.
\end{proof}

\section{Del Pezzo fibrations of degree $2$}

For integers $n,\lambda,\mu$ and $\nu$, we denote by $P = P_2 (n,\lambda,\mu,\nu)$ the $\mbP (1,1,1,2)$-bundle over $\mbP^{n-2}$ defined by
\[
\begin{pmatrix}
u_0 & \cdots & u_{n-2} & & x & y & z & w \\
1 & \cdots & 1 &  | & 0 & \lambda & \mu & \nu \\
0 & \cdots & 0 &  | & 1 & 1 & 1 & 2
\end{pmatrix}
\]
and consider the complete linear system $|\mcO_P (2 \nu,4)|$.
We assume that $n \ge 3$ and, by normalizing the action, we may assume $0 \le \lambda \le \mu$.
The aim of the present section is to prove the following.

\begin{Thm} \label{thm:dP2}
Suppose that the ground field is $\mbC$.
Let $X$ be a very general member of $|\mcO_P (2 \nu,4)|$, and suppose that $X$ is nonsingular and the projection $\pi \colon X \to \mbP^{n-2}$ is a del Pezzo fibration.
If $2 \nu - \lambda - \mu - (n-1) \ge 0$, then $X$ is not stably rational.
\end{Thm}

\begin{Cor} \label{cor:dP2}
Let $X$ be as in \emph{Theorem \ref{thm:dP2}} $($without assuming $2 \nu - \lambda - \mu - (n-1) \ge 0$$)$.
If either $n = 3$ or $-K_X$ is not ample, then $X$ is not stably rational.
\end{Cor} 

\begin{Lem} \label{lem:dP2num}
Suppose that the ground field is $\mbC$.
Let $X$ be a general member of $|\mcO_P (2 \nu,4)|$ and suppose that  $\pi \colon X \to \mbP^{n-2}$ is a nonsingular del Pezzo fibration.
Then the following hold.
\begin{enumerate}
\item $\nu \ge 1 $, and if $\nu = 1$, then $\lambda = \mu = 0$. 
\item $2 \nu \ge 3 \mu$.
\item $2 \nu \ge 4 \lambda$.
\item If $3 \mu < 2 \nu < 4 \mu$ and $2 \nu \ne 3 \mu + \lambda$, then $n = 3$.
\item If $2 \nu < 3 \mu + \lambda$, then $2 \nu = 3 \mu$.
\end{enumerate}
\end{Lem}

\begin{proof}
We prove (1).
Suppose that $\nu < 0$.
Then the monomials of bi-degree $(2 \nu,4)$ are divisible by $w$ and thus $X$ is either reducible or non-reduced.
Thus $\nu \ge 0$.
Suppose that $\nu = 0$. 
If in addition $\lambda > 0$, then $w^2, w x^2, x^4$ are the monomials of bi-degree $(2 \nu,4) = (0,4)$ and thus $X$ is reducible.
Thus $\lambda = 0$.
If $\mu > 0$, then 
\[
w^2, w x^2, w x y, w y^2, x^4, x^3 y, x^2 y^2, x y^3, y^4
\] 
are the monomials of bi-degree $(2 \nu,4) = (0,4)$.
In this case $X$ is singular along $(x = y = w = 0) \cong \mbP^1$.
This contradicts the smoothness of $X$ and thus $\mu = 0$.
Then $X$ is isomorphic to the product $\mbP^1 \times S$, where $S$ is a del Pezzo surface of degree $2$, and $\pi \colon X \to \mbP^1$ is not a del Pezzo fibration.
This proves the first assertion of (1).
The latter assertion follows (2).

We prove (2).
Suppose that $2 \nu < 3 \mu$.
Then there are no monomials of bi-degree $(2 \nu,4)$ which are                                                                                                                 divisible by $z^3$, and $z^2 w$ is of bi-degree $(2 \mu + \nu,4)$ with $2 \mu + \nu > 2 \nu$.
Thus $X$ is singular along $(x = y = w = 0) \cong \mbP^1$.
This is a contradiction and (2) is proved.

We prove (3).
Suppose that $2 \nu < 4 \lambda$.
Then, any monomial of bi-degree $(2 \nu,4)$ is divisible either $w$ or $x$, which implies that $X$ contains the sub WPS bundle $(x = w = 0) \subset P$.
This is impossible since $\rho (X) = 2$.
Proof of second assertion? 

We prove (4) and (5). 
We assume that $2 \nu < 4 \mu$.
Then the defining polynomial of $X$ can be written as
\[
z^3 (a_{2 \nu - 3 \mu} x + b_{2 \nu - 3 \mu - \lambda} y) + z^2 f (u,x,y,w) + z g (u,x,y,w) + h (u,x,y,w),
\]
where $a = a (u), b = b (u)$ are homogeneous polynomials of indicated degree and $f, g, h$ are of bi-degree $(2 \nu - 2 \mu,2)$, $(2 \nu - \mu,3)$, $(2 \nu,4)$, respectively.
Suppose in addition that $2 \nu < 3 \mu + \lambda$.
Then $b = 0$.
If $\deg a > 0$, then $X$ is singular along $(a = x = y = w = 0) \subset X$, which is non-empty.
Thus $\deg a = 0$, that is, $2 \nu = 3 \mu$, and (5) is proved.
Suppose that $2 \nu > 3 \lambda + \mu$.
Then $\deg a$ and $\deg b$ are positive.
It is easy to see that $X$ singular along $(a = b = x = y = w = 0)$, which is nonempty if $n \ge 4$.
Therefore we conclude $n = 3$ and (4) is proved.
\end{proof}

\begin{Lem} \label{lem:dP2ample}
Let the notation and assumption as in \emph{Lemma \ref{lem:dP2num}}.
\begin{enumerate}
\item If $-K_X$ is not ample, then $2 \nu - \lambda - \mu - (n-1) \ge 0$.
\item If $n = 3$, then $2 \nu - \lambda - \mu - 2 \ge 0$.
\end{enumerate}
\end{Lem}

\begin{proof}
Take $F_P \in |\mcO_P (1,0)|$, $D_P \in |\mcO_P (0,1)|$, and set $F = F_P|_X$, $D = D_P|_X$.
By adjunction, we have $-K_X \sim (n-1 + \lambda + \mu - \nu) F + D$. 

We prove (1).
Suppose first that $\nu \ge 2\mu$.
Then $|\mcO_P (\nu,2)|$ is base point free.
This implies that $\nu F + 2 D$ is nef.
Since $\rho (X) = 2$, a divisor $\alpha F + D$ is ample if $\alpha > \nu/2$, and the assertion follows immediately.
Suppose next that $\nu < 2\mu$.
Then $|\mcO_P (\mu,1)|$ is base point free.
This implies that $\lambda F + D$ is nef.
It follows that a divisor $\alpha F + D$ is ample if $\alpha > \lambda$, and we obtain the inequality $\lambda+ n-1 \le \nu$.
Combining this inequality with $2 \nu \ge 3\mu$, it is easy to check that the inequality $2 \nu \ge \mu+ \lambda + n-1$ holds.

We prove (2).
Assume to the contrary that $2 \nu - \lambda - \mu - 2 < 0$.
Then since $\lambda \le \mu$, we have $3 \mu \le 2 \nu \le 2 \mu + 1$.
Thus $\mu = 0,1$.
But this is impossible since $\nu > 0$.
\end{proof}

In the rest of the present section, we assume that $(n,\lambda,\mu,\nu)$ satisfies the assumption of Lemma \ref{lem:dP2num}.
We keep the following setting except in Subsection \ref{sec:dP2-3}.

\begin{Setting} \label{setting:dP2}
\begin{itemize}
\item The ground field $\K$ is algebraically closed and $\chara (\K) = 2$.
\item $X$ is a hypersurface of bi-degree $(2 \nu,4)$ in $P = P_2 (n,\lambda,\mu,\nu)$ defined by
\[
F := w^2 + f (u,x,y,z) = 0,
\]
where $f (u,x,y,z)$ is a general polynomial of bi-degree $(2 \nu,4)$.
\item $Z$ is the $\mbP^2$-bundle over $\mbP^{n-2}$ defined by
\[
\begin{pmatrix}
u_0 & \cdots & u_{n-2} & & x & y & z \\
1 & \cdots & 1 & | & 0 & \lambda & \mu \\
0 & \cdots & 0 & | & 1 & 1 & 1
\end{pmatrix}.
\]
\item $\mcL = \mcO_Z (\nu,2)$.
\item $\tau \colon X \to Z$ is the restriction of the projection $P \ratmap Z$, which is the double cover of $Z$ branched along $f \in H^0 (Z,\mcL^2)$.
\item $\mcM$ is the invertible sheaf on $X$ associated to the double cover $\tau$.
\end{itemize}
\end{Setting}

We define
\[
U_x = (x \ne 0) \subset Z, \ 
\Pi_y = (x = 0) \cap (y \ne 0) \subset Z, \ 
\Gamma_z = (x = y = 0) \subset Z,
\]
so that we have $Z = U_x \cup \Pi_y \cup \Gamma_z$.
We will analyze the critical points of $f \in H^0 (Z^{\circ},\mcL^2)$, which will be done in the following subsections.

\subsection{Case $\nu \ne 2 \mu$ and $\lambda,\mu,\nu$ does not satisfy $2 \nu = 3 \mu = 4 \lambda$} \label{sec:dP2-1}

\begin{Lem} \label{lem:dP3crit1}
The section $f \in H^0 (Z, \mcL^2)$ has only $($almost$)$ nondegenerate critical points on $Z$.
\end{Lem}

\begin{proof}
We first treat the case $\nu = 1$.
Then $\lambda = \mu = 0$ by Lemma \ref{lem:dP2num}.
Hence $Z \cong \mbP^{n-2} \times \mbP^2$ and $\mcL^2 \cong \mcO_{\mbP^{n-2} \times \mbP^2} (2,4)$.
We see that $2 = 2 \nu = \max \{2,2 \lambda,2\mu\}$.
If $n$ is even, then, by Lemma \ref{lem:surjrest}, a general $f \in H^0 (Z,\mcL^2)$ has only nondegenerate critical points on $Z$.
We assume that $n$ is odd.
Let $\msp \in Z$ be a point.
Replacing coordinates, we may assume $\msp = (1\!:\!0\!:\!\cdots\!:\!0 ; 1\!:\!0\!:\!0)$.
Then, since the 2nd and 3rd restriction maps of $\mcL^2$ at $\msp$ are surjective, $n$ independent conditions are imposed for a section in $H^0 (Z,\mcL^2)$ to have a critical point at $\msp$.
Consider a section of the form
\[
g = (u_1 u_2 + u_3 u_4 + \cdots + u_{n-4} u_{n-3}) x^4 + u_0 u_{n-2} x^3 y + u_0^2 x z^3 + \cdots,
\]
which has an almost nondegenerate critical point at $\msp$.  
This shows that additional conditions are imposed for a section to have a critical point at $\msp$ which is worse than almost nondegenerate critical point.
By counting dimension, we conclude that a general $f \in H^0 (Z,\mcL^2)$ has only almost nondegenerate critical points on $Z$.
In the rest of the proof, we assume that $\nu \ge 2$.

We have $2 \nu \ge \max \{3, 3 \lambda,3 \mu\}$ and, by Lemma \ref{lem:surjrest}, the 4th restriction map of $\mcL^2$ is surjective at any point $\msp \in U_x$, and it remains to consider critical points on $\Pi_y \cup \Gamma_z$.
We claim that $2 \nu \ge 4 \lambda + 1$.
Suppose that $2 \nu < 4 \lambda + 1$.
Then, by Lemma \ref{lem:dP2num}.(1), we have $2 \nu = 4 \lambda$.
Since we are assuming that $2 \nu = 3 \mu = 4 \lambda$ does not hold in this subsection, we have $2 \nu \ne 3 \mu$, which implies $4 \lambda = 2 \nu \ge 3 \mu + \lambda$ by Lemma \ref{lem:dP2num}.(5).
This implies $\lambda \ge \mu$, hence $\lambda = \mu$.
This is a contradiction since we are assuming $\nu \ne 2 \mu$. 
We divide the proof into several cases.

Suppose that $2 \nu \ge 3 \lambda + \mu$.
Then, by Lemma \ref{lem:surjrest}, the 2nd restriction map of $\mcL^2$ is surjective at any point $\msp \in \Pi_y$.
Thus a general $f \in H^0 (Z, \mcL^2)$ has only (almost) nondegenerate critical points on $Z \setminus \Gamma_z$ and it remains to consider critical point along $\Gamma_z$.
If $2 \nu > 4 \mu$, then we are done by Lemma \ref{lem:surjrest}.
In the following we show that $f$ does not have a critical point along $\Gamma_z$ assuming that $2 \nu < 4 \mu$.
In this case we can write
\[
f = a_{2 \nu - 3 \mu} x z^3 + b_{2 \nu - 3 \mu - \lambda} y z^3 + g,
\]
where $g (u,x,y,z) \in (x,y)^2$.
This means that the set of critical points of $f$ contained in $\Gamma_z$  is contained in $C := (a = b = 0) \cap \Gamma_z$.
Clearly $C = \emptyset$ if either $a$ or $b$ is a constant, that is, either $2 \nu = 3 \mu$ or $2 \nu = 3 \mu + \lambda$.
If $\nu \ne 3 \mu, 3 \mu + \lambda$, then $n = 3$ and the $C$ is also empty because $a$ and $b$ are general.
Thus $f$ cannot have a critical point on $\Gamma_z$.

Suppose that $2 \nu < 3 \lambda + \mu$.
Then we have $2 \nu = 3 \mu$ by Lemma \ref{lem:dP2num}.
We can write
\[
f = a_{2 \nu - 4 \lambda} y^4 + x (z^3 + b_{\mu - \lambda} z^2 y + c_{2 \mu - 2 \lambda} z y^2 + d_{3 \mu - 3 \lambda} y^3) + g,
\]
where $a,b,c,d$ are homogeneous polynomials in variables $u$ of the indicated degree and $g = g (u,x,y,z) \in (x^2)$ is of bi-degree $(2 \nu, 4)$.
Note that $\deg a = 2 \nu - 4 \lambda \ge 2$.
We set $h = z^3 + b z^2 y + c z y^2 + d y^3$.
It is easy to check that
\[
\Crit (f) \cap (x = 0) = C \cap (x = h = 0) \subset Z,
\]
where
\[
C = \left( \frac{\prt a}{\prt u_0} = \cdots = \frac{\prt a}{\prt u_{n-2}} = 0 \right) \subset Z.
\] 
The set $C$ is a union of fibers of the $\mbP^2$-bundle $Z \to \mbP^{n-2}$ over critical points of $a$, viewed as a section in $H^0 (\mbP^{n-2}, \mcO_{\mbP^{n-2}} (2 \nu - 4 \lambda))$.
Since $a$ is general, we see that $C$ consists of finitely many fibers and thus $\Crit (f)$ consists of finitely many closed points.

Now we consider the case when $n$ is odd and $\deg a = 2$.
In this case we claim $C = \emptyset$.
Indeed, then the Hessian matrix of the quadric in even number of variables $u_0,\dots,u_{n-2}$ is invertible, and this implies that the set of critical points of $a \in H^0 (\mbP^{n-2}, \mcO_{\mbP^{n-2}} (2))$ is empty. 
Thus we have $\Crit (f) \cap (x = 0) = \emptyset$.

In the following we assume that either $n$ is even or $n$ is odd and $\deg a > 2$. 
Then we see that the section $a$, viewed as a section on $\mbP^{n-2}$, has only (almost) nondegenerate critical points. 
Moreover we may assume that $\prt h/\prt z$ does not vanish at any point of $\Crit (f) \cap (x = 0)$.
This is indeed possible by choosing $c$ so that $C \cap (c = 0) = \emptyset$.
Now let $\msp \in \Crit (f) \cap (x = 0)$ be a point.
We may assume that $\msp = (1\!:\!0\!:\!\cdots\!:\!0; 0\!:\!1\!:\!0) \in Z$.
We can choose $u_1,\dots,u_{n-2}, x, z$ as local coordinates of $Z$ at $\msp$.
By using these coordinates, the restriction of $f$ to $\msp$ is of the form 
\[
q (u_1,\dots,u_{n-2}) + x (z + \ell (u_1,\dots,u_{n-2}) + e_2) + x^2 (\alpha + e_1),
\]
where
\[
q =
\begin{cases}
u_1 u_2 + u_3 u_4 + \cdots + u_{n-3} u_{n-2}, & \text{if $n$ is even}, \\
u_1^2 + u_2 u_3 + \cdots + u_{n-3} u_{n-2} + u_1^3, & \text{if $n$ is odd},
\end{cases}
\]
$\alpha \in \K$, $\ell$ is a linear form in $u_1,\dots,u_{n-2}$ and $e_i = e_i (u_1,\dots,u_{n-2},x,z) \in \mfm_{\msp}^i$, where $\mfm_{\msp} = (u_1,\dots,u_{n-2},x,z)$.
It is now clear that $f$ has an (almost) nondegenerate critical points at $\msp$ and the proof is completed.
\end{proof}

\subsection{Case $\nu = 2 \mu$} \label{sec:dP2-2}

\subsubsection{Subcase $2 \nu = 4 \mu$ and $\mu > \lambda$}

\begin{Lem} \label{lem:dP3crit2}
The section $f \in H^0 (Z, \mcL^2)$ has only $($almost$)$ nondegenerate critical points on $Z \setminus \Gamma_z = U_x \cup \Pi_y$.
\end{Lem}

\begin{proof}
Note that we have $\nu \ge 2$ (and $\mu \ge 1$).
Hence $2 \nu \ge \max \{3, 3 \lambda, 3 \mu\}$ and $2 \nu \ge \max \{4 \lambda + 1,3 \lambda + \mu\}$ and the assertion follows from Lemma \ref{lem:surjrest}.
\end{proof}

Let $\Crit (f) \subset Z$ be the set of critical points of $f$ and set $C_1 = \Crit (f) \cap (Z \setminus \Gamma_z)$ and $C_2 = \Crit (f) \cap \Gamma_z$, so that we have $\Crit (f) = C_1 \cup C_2$.
By Lemma \ref{lem:dP3crit2}, $C_1$ consists of isolated points which are (almost) nondegenerate critical points.
We need to analyze $C_2$ and construct a resolution of singularities of $X$ along $\tau^{-1} (C_2)$.

We can write
\[
f = \alpha z^4 + a_{\mu} x z^3 + b_{\mu - \lambda} y z^3 + g (u,x,y,z),
\]
where $a_{\mu}, b_{\mu - \lambda}$ are polynomials in variables $u$ of indicated degree and $g (u,x,y,z) \in (x,y)^2$ is of bi-degree $(4 \mu,4)$. 
It is easy to see that $C_2 = (x = y = a_{\mu} = b_{\mu - \lambda} = 0)$ and $C_2$ is covered by $U_{i,z} = (u_i \ne 0) \cap (z \ne 0) \subset Z$ for $i = 0,\dots,n-2$.
In the following we work with $U_{0,y}$ and analyze the restriction of $f$ to $U_{0,z}$.
The analysis for the other $U_{i,z}$ is completely the same.

The open set $U_{0,z}$ is isomorphic to $\mbA^n$ with affine coordinates $u_1,\dots,u_{n-2},x,y$ and we have
\[
f|_{U_{0,z}} = \alpha + a_{\mu} x + b_{\mu - \lambda} y + \bar{g} \sim_2 a_{\mu} x + b_{\mu - \lambda} y + \bar{g},
\] 
where $\bar{g} = g (1,u_1,\dots,u_{n-2},x,y,1)$.
Since $a, b$ are general and $\bar{g} \in (x,y)^2$, the polynomial $f|_{U_{0,z}}$ satisfies Condition \ref{cond:resolposcrit2}.
Thus, by Lemma \ref{lem:resol2}, the blowup of $X$ along $\tau^{-1} (C_2)$ is universally $\CH_0$-trivial, resolves the singularity of $X$ along $\tau^{-1} (C_2)$ and pulls back $\mcM$ into a subsheaf of $\Omega^{n-1}$.

\subsubsection{Subcase $2 \nu = 4 \mu$ and $\mu = \lambda$}

\begin{Lem} \label{lem:dP3crit3}
The section $f \in H^0 (Z, \mcL^2)$ has only $($almost$)$ nondegenerate critical points on $U_x$.
\end{Lem}

\begin{proof}
Note that we have $\nu \ge 2$ (and $\mu \ge 1$).
Hence $2 \nu \ge \max \{3, 3 \lambda, 3 \mu\}$ and the assertion follows from Lemma \ref{lem:surjrest}.
\end{proof}

Let $\Crit (f) \subset Z$ be the set of critical points of $f$ and set $C_1 = \Crit (f) \cap U_x$ and $C_2 = \Crit (f) \cap (Z \setminus U_x)$, so that we have $\Crit (f) = C_1 \cup C_2$.
By Lemma \ref{lem:dP3crit3}, $C_1$ consists of isolated points which are (almost) nondegenerate critical points.
We need to analyze $C_2$ and construct a resolution of singularities of $X$ along $\tau^{-1} (C_2)$.

We can write
\[
f = \alpha y^4 + \beta y^3 z + \gamma y^2 z^2 + \delta y z^3 + \varepsilon z^4 + x (a y^3 + b y^2 z + c y z^2 + d z^3) + g,
\]
where $\alpha,\dots,\varepsilon \in \K$, $a,\dots,d$ are polynomials in variables of degree $\mu > 0$ and $g = g (u,x,y,z) \in (x^2)$ if of bi-degree $(4 \mu,4)$.
Replacing $z$ and rescaling $y$, we assume $\delta = 0$ and $\beta = 1$.
We write $g = x^2 z^2 e + h$, where $e = e (u)$ and $h = h (u,x,y,z) \in (x^2) \cap (x,y)^3$.
It is then easy to see that $C_2 = (x = y = d = 0)$ and $C_2$ is covered by the open sets $U_{i,z} = (u_i \ne 0) \cap (z \ne 0) \subset Z$ for $i = 0,\dots,n-2$.
In the following we work with $U_{0,z}$ and analyze the restriction of $f$ to $U_{0,z}$.

The open subset $U_{0,z}$ is isomorphic to $\mbA^n$ with affine coordinates $u_1,\dots,u_{n-2},x,y$ and we have
\[
f|_{U_{0,z}} \sim_2 \bar{d} x + \bar{e} x^2 + \bar{c} x y + y^3 + (\bar{a} x y^3 + \bar{b} x y^2 + \bar{h}).
\]
The hypersurface in $\mbA^{n-2}_u$ defined by $\bar{d} = 0$ is nonsingular since $d$ is general, and $h \in (x^2) \cap (x,y)^3$.
Thus, by Lemma \ref{lem:resol2}, the blowup of $X$ along $\tau^{-1} (C_2)$ is universally $\CH_0$-trivial, resolves the singularity of $X$ along $\tau^{-1} (C_2)$ and pulls back $\mcM$ into a subsheaf of $\Omega^{n-1}$.

The following is the conclusion of Sections \ref{sec:dP2-1} and \ref{sec:dP2-2}. 

\begin{Prop} \label{prop:exitCHresoldP2}
Suppose that $(n,\lambda,\mu,\nu)$ satisfies the assumption of \emph{Lemma \ref{lem:dP2num}} and let $X$ be as in \emph{Setting \ref{setting:dP2}}.
Then there exists a universally $\CH_0$-trivial resolution $\varphi \colon Y \to X$ of singularities of $X$ such that $\varphi^*\mcM \inj \Omega_Y^{n-1}$, and we have
\[
\mcM \cong \mcO_X (2 \nu - \lambda - \mu - (n-1), 1).
\]
\end{Prop}

\begin{proof}
The existence of $\varphi \colon Y \to X$ follows from the results of Sections \ref{sec:dP2-1}, \ref{sec:dP2-2} and Proposition \ref{prop:existresol}.
We have $\omega_Z \cong \mcO_Z (-(n-1)-\lambda-\mu,-3)$ and $\mcL \cong \mcO_Z (\nu,2)$.
Thus
\[
\mcM \cong \tau^* (\omega_Z \otimes \mcL^2) \cong \mcO_X (2 \nu - \lambda - \mu - (n-1),1),
\]
and the proof is completed.
\end{proof}

\subsection{Case $2 \nu = 3 \mu = 4 \lambda$} \label{sec:dP2-3}

In this subsection we keep the following setting.

\begin{Setting} \label{setting:dP2-3}
\begin{itemize}
\item The ground field $\K$ is algebraically closed and $\chara (\K) = 3$.
\item $X$ is a hypersurface in $P = P_2 (n,\lambda,\mu,\nu)$ defined by
\[
F := z^3 x + f (u,x,y,w) = 0,
\]
where $f (u,x,y,w)$ is a general polynomial of bi-degree $(2 \nu,4)$.
\item $R$ is the $\mbP (1,1,3,2)$-bundle over $\mbP^{n-2}$ defined by
\[
\begin{pmatrix}
u_0 & \cdots & u_{n-2} & & x & y & \bar{z} & w \\
1 & \cdots & 1 & | & 0 & \lambda & 3 \mu & \nu \\
0 & \cdots & 0 & | & 1 & 1 & 3 & 2
\end{pmatrix}.
\]
\item $Z$ is the hypersurface in $R$ defined by $\bar{F} := \bar{z} x + f = 0$.
\item $\mcL = \mcO_Z (\mu,1)$.
\item $\tau \colon X \to Z$ is the restriction of the triple cover $P \to R$, which is a triple cover of $Z$ branched along $\bar{z} \in H^0 (Z,\mcL^3)$.
\item $Z^{\circ} = Z \setminus (x = y = w = 0) \cap Z$, $X^{\circ} = \tau^{-1} (Z^{\circ})$ and $\tau^{\circ} = \tau|_{X^{\circ}}$.
\end{itemize}
\end{Setting}

\begin{Lem} \label{lem:nonsingZdP2-3}
$Z^{\circ}$ is nonsingular and $X$ is nonsingular along $X \setminus X^{\circ}$.
\end{Lem}

\begin{proof}
It is clear that the singular locus of $Z^{\circ}$ is contained in $(x = 0) \cap Z = (x = f = 0)$.
We can write the defining polynomial of $Z$ as
\[
\bar{F} = \bar{z} x + \alpha w^2 + w (\beta y^2 + g) + \gamma y^4 + h,
\]
where $\alpha, \beta,\gamma \in \K$ and $h = h (u,x,y)$, $g = g (u,x,y)$ are polynomials of bi-degree $(\nu,2)$, $(2 \nu,4)$, respectively, and both of them are divisible by $x$.
Since $f$ is general, $\alpha,\beta,\gamma$ are general.
In particular $\alpha \ne 0$ and we have $Z^{\circ} = Z \setminus \big( (x = y = 0) \cap Z \big)$, so that $Z^{\circ}$ is covered by two open subsets $U_x = (x \ne 0)$ and $U_y = (y \ne 0)$ of $R$.
By the above argument, $Z \cap U_x$ is nonsingular and it remains to  show that $Z \cap U_y$ is nonsingular along $(x = 0) \cap U_y$.
The open set $U_y \subset R$ is isomorphic to $\mbP^{n-2} \times \mbA^3_{x,\bar{z},w}$ and $Z \cap U_y$ is defined by
\[
\bar{F} (u,x,1,\bar{z},w) = \bar{z} x + \alpha w^2 + w (\beta + g (u,x,1)) + \gamma + h (u,x,1) = 0.
\]
We have $(\prt \bar{F}/\prt w)|_{(x = 0)} = 2 \alpha w + \beta$ since $g, h$ are divisible by $x$. 
It is then straightforward to check that the singular locus of $Z \cap U_y$ along $(x = 0) \subset U_y$ is contained in the set
\[
(x = \bar{F} = 2 \alpha w + \beta = 0) \subset U_y,
\] 
which is empty since $\alpha, \beta,\gamma$ are general.
Thus $Z^{\circ}$ is nonsingular.

We have $X \setminus X^{\circ} = (x = y = w = 0) \cap X$ which is contained in the open subset $V = (z \ne 0) \subset P$.
The open subset $V$ is isomorphic to $\mbP^{n-2} \times \mbA^3_{x,y,w}$ and $X \cap V$ is the hypersurface defined by $x + f (u,x,y,w) = 0$ in $V$.
It is then straightforward to check that $X \cap V$ is nonsingular along $(x = y = w = 0) \subset V$.
Thus $X$ is nonsingular along $X \setminus X^{\circ}$ as desired.
\end{proof}

Since $Z^{\circ}$ is nonsingular, we can define $\mcM^{\circ}$ to be the invertible sheaf on $X^{\circ}$ associated to the triple cover $\tau^{\circ}$, and we set $\mcM = \iota_* \mcM^{\circ}$, where $\iota \colon X^{\circ} \inj X$.

\begin{Lem} \label{lem:critdP2case3}
The section $\bar{z}$ has only nondegenerate critical points on $Z^{\circ}$.
\end{Lem}

\begin{proof}
Since $\bar{z} x$ is the unique term of the defining polynomial $\bar{z} x + f$ of $Z$ involving $\bar{z}$, we can choose $\bar{z}$ (or more precisely its translation) as a part of local coordinates of $Z$ at any point in $(x = 0) \cap Z$.
Hence the section $\bar{z}$ does not have a critical point along $(x = 0) \cap Z$.

It remains to consider critical points of $\bar{z}$ on the open set $U_x$.
On $U_x$, by setting $x = 1$, we have $\bar{z} = - f$.
Eliminating $\bar{z}$, it is enough to show that $f$ has only nondegenerate critical points on the open subset $(x \ne 0)$ of the $\mbP (1,1,2)$-bundle $R'$ over $\mbP^{n-2}$ defined by
\[
\begin{pmatrix}
u_0 & \cdots & u_{n-2} & & x & y & w \\
1 & \cdots & 1 & | & 0 & \lambda & \nu \\
0 & \cdots & 0 & | & 1 & 1 & 2
\end{pmatrix}.
\]
The section $f$ is a general element of $H^0 (R', \mcN)$, where $\mcN = \mcO_{R'} (2 \nu,4)$ and, by Lemma \ref{lem:surjrest}, the $3$rd restriction map
\[
r_3 (\msp) \colon H^0 (R', \mcN) \to \mcN \otimes (\mcO_{R'}/\mfm_{\msp}^3)
\]
is surjective at any point $\msp \in (x \ne 0) \subset R'$ since $2 \nu \ge \max \{2,2 \lambda,2\nu\}$.
This shows that $f$ has only nondegenerate critical points on $(x \ne 0) \subset R'$ and the proof is completed.
\end{proof}

\begin{Prop} \label{prop:exitCHresoldP2-3}
Suppose that $(n,\lambda,\mu,\nu)$, where $2 \nu = 3 \mu = 4 \lambda$, satisfies the assumption of \emph{Lemma \ref{lem:dP2num}} and let $X$ be as in \emph{Setting \ref{setting:dP2-3}}.
Then there exists a universally $\CH_0$-trivial resolution $\varphi \colon Y \to X$ of singularities of $X$ such that $\varphi^*\mcM \inj \Omega_Y^{n-1}$, and we have
\[
\mcM \cong \mcO_X (\nu - \lambda - (n-1),0).
\]
\end{Prop}

\begin{proof}
The existence of $\varphi \colon Y \to X$ follows from Lemmas \ref{lem:nonsingZdP2-3}, \ref{lem:critdP2case3} and Proposition \ref{prop:existresol}.
We have $\omega_Z \cong \mcO_Z (-(n-1) - \lambda - 3 \mu + \nu, -3)$ and $\mcL \cong \mcO_Z (\mu,1)$, so that
\[
\mcM \cong \tau^* (\omega_Z \otimes \mcL^3) \cong \mcO_Z (- (n-1) - \lambda + \nu,0),
\]
and the proof is completed.
\end{proof}

\subsection{Proofs of Theorem \ref{thm:dP2} and Corollary \ref{cor:dP2}}

\begin{proof}[Proof of \emph{Theorem \ref{thm:dP2}}]
Note that $\nu - \lambda - (n-1) \ge 2 \nu - \lambda - \mu - (n-1)$ in the case when $2 \nu = 3 \mu = 4 \lambda$.
Thus, by the same reasoning as in the proof of Theorem \ref{thm:dP1}, this follows from Propositions \ref{prop:exitCHresoldP2}, \ref{prop:exitCHresoldP2-3} and Theorem \ref{thm:spCH0}.
\end{proof}

\begin{proof}[Proof of \emph{Corollary \ref{cor:dP2}}]
This follows from Lemma \ref{lem:dP2ample} and Theorem \ref{thm:dP2}
\end{proof}

\section{Del Pezzo fibrations of degree $3$}

For integers $n \ge 3, \lambda,\mu,\nu$, we denote by $P_3 (n,\lambda,\mu,\nu)$ the $\mbP^3$-bundle over $\mbP^{n-2}$ defined by
\[
\begin{pmatrix}
u_0 & \cdots & u_{n-2} & & x & y & z & w \\
1 & \cdots & 1 &  | & 0 & \lambda & \mu & \nu \\
0 & \cdots & 0 &  | & 1 & 1 & 1 & 1
\end{pmatrix},
\]
and we consider the complete linear system $|\mcO_P (\theta,3)|$ for an integer $\theta$, where $P = P_3 (n,\lambda,\mu,\nu)$.
For a member $X \in |\mcO_P (\theta,3)|$, we denote by $\pi \colon X \to \mbP^{n-2}$ the restriction of the projection $P \to \mbP^{n-2}$, whose fibers are cubic hypersurfaces in $\mbP^3$.
The aim of the present section is to prove the following.

\begin{Thm} \label{thm:dP3}
Suppose that the ground field is $\mbC$.
Let $X$ be a very general member of $|\mcO_P (\theta,3)|$ and suppose that $\pi \colon X \to \mbP^{n-2}$ is a nonsingular del Pezzo fibration.
If $\theta \ge \lambda + \mu + (n-1)$, then $X$ is not stably rational.
\end{Thm}

\begin{Cor} \label{cor:dP3}
Let $X$ be as in \emph{Theorem \ref{thm:dP3}} $($without assuming $\theta \ge \lambda + \mu + (n-1)$$)$.
If either $-K_X$ is not ample or $n = 3$ and $(\theta,\lambda,\mu,\nu) \ne (1,0,0,0), (3,1,1,1)$, then $X$ is not stably rational.
\end{Cor}

\begin{Lem} \label{lem:dP3num}
Suppose that the ground filed is $\mbC$.
Let $X$ be a general member of $|\mcO_P (\theta,3)|$ and suppose that $\pi \colon X \to \mbP^{n-2}$ is a nonsingular del Pezzo fibration.
Then the following hold.
\begin{enumerate}
\item $\theta \ge 2 \nu$.
\item $\theta \ge 3 \mu$.
\item If $2 \nu + \mu < \theta < 3 \nu$, then $n \le 4$.
\item If $2 \nu + \lambda < \theta < 2 \nu + \mu$, then $n = 3$.
\item If $\theta < 2 \nu + \lambda$, then $\theta = 2 \nu$.
\item If $\theta \le 2$, then the quadruple $(\theta,\lambda,\mu,\nu)$ is one of the following:
\[
(2,0,0,0), (2,0,0,1), (1,0,0,0).
\]
\end{enumerate}
\end{Lem}

\begin{proof}
Clearly $\theta \ge 2 \nu$ because otherwise $X$ is singular along $(x = y = z = 0)$.
If $\theta < 3 \mu$, then $X$ contains the $\mbP^2$-bundle $(x = y = 0) \subset P$ over $\mbP^{n-2}$ and thus the Picard number of $X$ is at least $3$.
These prove (1) and (2).

We prove (3), (4) and (5).
Suppose that $\theta < 3 \nu$.
Then $X$ is defined by an equation of the form
\[
w^2 (a_{\theta - 2 \nu} x + b_{\theta - 2 \nu - \lambda} y + c_{\theta-2 \nu - \mu} z) + w f (u,x,y,z) + g (u,x,y,z) = 0,
\]
where $a, b, c$ are homogeneous polynomials in variables $u$ of indicated degrees and $f,g$ are of bi-degrees $(\theta - \nu,2)$, $(\theta,3)$, respectively.
The set $\Sigma = (x = y = z = a = b = c = 0)$ is contained in the singular locus of $X$.
Hence we must have $\Sigma = \emptyset$, that is, $n \le 4$ if $2 \nu + \mu < \theta$, and $n = 3$ if $2 \nu + \lambda < \theta < 2 \nu + \mu$.
This proves (3) and (4).
If $\theta < 2 \nu + \lambda$, then $b = c = 0$ and the set $\Sigma$ is empty if and only if $a$ is a constant, that is, $\theta = 2 \nu$.
This proves (5).

Finally we prove (6).
We claim that $\theta \ge 1$.
Clearly we have $\theta \ge 0$ because otherwise $|\mcO_P (\theta,3)| = \emptyset$.
If $\theta = 0$, then $\lambda = \mu = \nu = 0$ by (1) and thus $X$ is the product $\mbP^{n-2} \times S$, where $S$ is a cubic surface.
This is impossible and we have $\theta \ge 1$ and the claim is proved.
Now we assume that $\theta \le 2$.
By (2), we have $\lambda = \mu  = 0$.
If $\theta = 1$, then $\nu = 0$ by (1) and we have $(\theta,\lambda,\mu,\nu) = (1,0,0,0)$.
If $\theta = 2$, then $\nu \le 1$ by (1) and we have $(\theta,\lambda,\mu,\nu) = (2,0,0,0), (2,0,0,1)$.
This completes the proof.
\end{proof}

\begin{Lem} \label{lem:dP3ample}
Suppose that the ground filed is $\mbC$ and $\pi \colon X \to \mbP^{n-2}$ is a nonsingular del Pezzo fibration.
If either $-K_X$ is not ample or $n = 3$ and $(\theta,\lambda,\mu, \nu) \ne (1,0,0,0), (3,1,1,1)$, then $\theta \ge \lambda + \mu + (n-1)$.
\end{Lem}

\begin{proof}
Take divisors $F_P \in |\mcO_Q (1,0)|$, $D_P \in |\mcO_P (0,1)|$ and set $F = F_P|_X$, $D = D_P|_X$.
It is clear that $\mcO_P (\nu,1)$ is generated by global sections and not ample.
Thus the cone of ample divisors on $P$ is the interior of the cone spanned by $F_P$ and $\nu F_P + D_P \in |\mcO_P (\nu,1)|$.
Hence a divisor $\alpha F + D$ on $X$ is ample if $\alpha > \nu$.
By adjunction, we have an isomorphism 
\[
\mcO_X (-K_X) \cong \mcO_X (n-1 + \lambda + \mu + \nu - \theta, 1), 
\]
that is, $-K_X \sim (n-1 + \lambda + \mu + \nu - \theta) F + D$.
Thus, if $-K_X$ is not ample, then the inequality $\theta \ge \lambda + \mu + (n-1)$ holds.

We consider the case when $n = 3$ and we assume that $\theta \le \lambda + \mu + 1$.
Since $3 \mu \le \theta$ and $\lambda \le \mu$, we have $3 \mu \le 2 \mu + 1$, that is, $\mu = 0,1$.
Suppose that $\mu = 0$.
In this case $\lambda = 0$ and $\theta \le 1$, which implies $\theta = 1$ since $\theta > 0$.
Moreover we have $\nu = 0$ since $2 \nu \le \theta$ and thus $(\theta,\lambda,\mu,\nu) = (1,0,0,0)$.
Suppose that $\mu = 1$.
By the inequalities $\lambda \le \mu \le \nu$, $\theta \ge 3 \mu$, $\theta \le \lambda + \mu + 1$ and $\theta \ge 2 \nu$, we have $\lambda = \nu = 1$ and $\theta = 3$, that is, $(\theta,\lambda,\mu,\nu) = (3,1,1,1)$.
Therefore the proof is completed.
\end{proof}

\begin{Rem}
If $(\theta,\lambda,\mu,\nu) = (1,0,0,0)$, then $X$ is a hypersurface of bi-degree $(1,3)$ on $\mbP^{n-2} \times \mbP^3$ and it is clearly rational.
If $(\theta,\lambda,\mu,\nu) = (3,1,1,1)$, then $X$ is birational to a (nonsinglar) cubic $n$-fold.
More precisely $X$ is the blowup of a nonsingular cubic $n$-fold along a nonsingular plane cubic curve.
\end{Rem}

In the following subsections we assume that $(\theta,\lambda,\mu,\nu)$ satisfies the assumption of Lemma \ref{lem:dP3num} and that $(\theta,\lambda,\mu,\nu) \ne (1,0,0,0)$.
Note that we have $\theta \ge 2$ by Lemma \ref{lem:dP3num}.

\subsection{Case $\theta > 3 \nu$} \label{sec:dP3case1}

In this section we keep the following setting.

\begin{Setting} \label{setting:dP3-1}
\begin{itemize}
\item The ground field $\K$ is algebraically closed and $\chara (\K) = 3$.
\item $X$ is a hypersurface in $P = P_3 (n,\lambda,\mu,\nu)$ defined by
\[
F := a_{\theta - 3 \nu} w^3 + f (u,x,y,z) = 0,
\]
where $a_{\theta - 3 \nu} = a_{\theta - 3 \nu} (u)$ is homogeneous of degree $\theta - 3 \nu > 0$ and $f (u,x,y,z)$ is of bi-degree $(\theta,3)$. 
We assume that $a$ and $f$ are both general.
\item $R$ is the $\mbP (1,1,1,3)$-bundle over $\mbP^{n-2}$ defined by
\[
\begin{pmatrix}
u_0 & \cdots & u_{n-2} & & x & y & z & \bar{w} \\
1 & \cdots & 1 &  | & 0 & \lambda & \mu & 3 \nu \\
0 & \cdots & 0 &  | & 1 & 1 & 1 & 3
\end{pmatrix}.
\]
\item $Z$ is the hypersurface in $R$ defined by $\bar{F} :=  a \bar{w} + f = 0$.
\item $\mcL = \mcO_Z (\nu,1)$.
\item $\tau \colon X \to Z$ is the restriction of the triple cover $P \to R$, which is a triple cover of $Z$ branched along $\bar{w} \in H^0 (Z, \mcL^3)$.
\item $Z^{\circ} = Z \setminus (x = y = z = 0) \cap Z$, $X^{\circ} = \tau^{-1} (Z^{\circ})$ and $\tau^{\circ} = \tau|_{X^{\circ}}$.
\end{itemize}
\end{Setting}

\begin{Lem} \label{lem:nonsingZdP3}
$Z^{\circ}$ is nonsingular and $X$ is nonsingular along $X \setminus X^{\circ}$.
\end{Lem}

\begin{proof}
It is clear that the singular locus of $Z^{\circ}$ is contained in $(a = f = 0)$. 
We will show that $(\bar{w} = a = f = 0) \subset R$ is nonsingular, which implies that $Z^{\circ}$ is nonsingular by Lemma \ref{lem:crismhyp}. 
We see that $(\bar{w} = 0) \subset R$ is isomorphic to the $\mbP^2$-bundle $Q$ over $\mbP^{n-2}$ defined by
\[
\begin{pmatrix}
u_0 & \cdots & u_{n-2} & & x & y & z \\
1 & \cdots & 1 & | & 0 & \lambda & \mu \\
0 & \cdots & 0 & | & 1 & 1 & 1
\end{pmatrix},
\]
which is also isomorphic to $\mbP_{\mbP^{n-2}} (\mcE)$, where $\mcE$ is a direct sum of three invertible sheaves on $\mbP^{n-2}$.
Let $H$ be the hypersurface in $Q$ defined by $a = 0$.
Then it is isomorphic to $\mbP_{H'} (\mcE|_{H'})$, where $H'$ is the hypersurface in $\mbP^{n-2}$ defined by $a = 0$.
The hypersurface $H' \subset \mbP^{n-2}$ is nonsingular since $a$ is general, and hence $H$ is nonsingular. 
We see that $(\bar{w} = a = f = 0) \subset Q$ is isomorphic to the hypersurface in $H$ defined by $f = 0$.
The section $f$ can be viewed as a section of $\mcO_{Q} (\theta,3)$ which is very ample since $\theta > \nu$.
It follows that $f|_V$ is a general section of the very ample invertible sheaf $\mcO_{Q} (\theta,3)|_H$ and $(\bar{w} = a = f = 0)$ is nonsingular.

We prove the latter part.
We have $X \setminus X^{\circ} = (x = y = z = 0) \subset P$.
It is easy to see that $X$ is nonsingular along $(x = y = z = 0)$ if the hypersurface in $\mbP^{n-2}$ defined by $a = 0$ is nonsingular.
The latter is clearly satisfied since $a$ is general.
\end{proof}

Since $Z^{\circ}$ is nonsinguar, we can define $\mcM^{\circ}$ to be the invertible sheaf on $X^{\circ}$ associated to the triple cover $\tau^{\circ}$, and we set $\mcM = \iota_*\mcM^{\circ}$, where $\iota \colon X^{\circ} \inj X$.

\begin{Lem} \label{lem:critdP1case1}
The section $\bar{w}$ has only nondegenerate critical points on $Z^{\circ}$.
\end{Lem}

\begin{proof}
Clearly $\bar{w}$ does not have a critical point on $(a = 0) \cap Z^{\circ}$.
On the open subset $(a \ne 0) \cap Z^{\circ}$, the section $\bar{w}$ has a critical point if and only if $- a^3 \bar{w} = a^2 f$ has a critical point (see Remark \ref{rem:critunit}).
Let $Q$ be the $\mbP^2$-bundle over $\mbP^{n-2}$ defined in the proof of Lemma \ref{lem:nonsingZdP3}, which is isomorphic to $(\bar{w} = 0) \subset R$.
It is enough to show that the section $a^2 f$, viewed as a section on $Q$, has only nondegenerate critical points on the open set $U = (a \ne 0) \subset Q$.
We define
\[
V = \{ a^2 g \mid g \in H^0 (Q, \mcO_{Q} (\theta,3)) \} \subset H^0 (Q, \mcO_{Q} (3 \theta - 6 \nu, 3)).
\]
which is a $\K$-vector space (note that we are fixing $a$).
For a point $\msp \in U$ and an integer $k > 0$, we consider the restriction maps
\[
r_{V, k} (\msp) \colon V \to \mcO_{Q} (3 \theta - 6 \nu,3)) \otimes (\mcO_{Q}/\mfm^k_{\msp}),
\]
\[
r_k (\msp) \colon H^0 (Q, \mcO_{Q} (\theta,3)) \to \mcO_{Q} (\theta,3) \otimes (\mcO_{Q}/\mfm_{\msp}^k).
\]
If $r_k (\msp)$ is surjective, then so is $r_{V,k} (\msp)$ since $a$ does not vanish at $\msp \in U$.
Thus, by the dimension counting argument, it is enough to show that $r_3 (\msp)$ and $r_2 (\msp)$ are surjective at any point $\msp \in (x \ne 0) \subset Q$ and $\msp \in (x = 0) \subset Q$, respectively. 
But this follows from Lemma \ref{lem:surjrest} since $\theta \ge \max \{2, 2 \lambda, 2 \mu\}$ and $\theta > 3 \mu$, and the proof is completed.
\end{proof}

\subsection{Case $\theta = 3 \nu$} \label{sec:dP3case2}

In this subsection we keep the following setting.

\begin{Setting} \label{setting:dP3-2}
\begin{itemize}
\item The ground field $\K$ is algebraically closed and of $\chara (\K) = 3$.
\item $X$ is a hypersurface in $P = P_3 (n,\lambda,\mu,\nu)$ defined by
\[
F := w^3 + f (u,x,y,z) = 0,
\]
where $f (u,x,y,z)$ is a general polynomial of bi-degree $(\theta,3)$. 
\item $Z$ is the $\mbP^2$-bundle over $\mbP^{n-2}$ defined by
\[
\begin{pmatrix}
u_0 & \cdots & u_{n-2} & & x & y & z \\
1 & \cdots & 1 &  | & 0 & \lambda & \mu \\
0 & \cdots & 0 &  | & 1 & 1 & 1
\end{pmatrix}.
\]
\item $\mcL = \mcO_Z (\nu,1)$.
\item $\tau \colon X \to Z$ is the restriction of the projection $P \ratmap Z$, which is a triple cover of $Z$ branched along $f \in H^0 (Z, \mcL^3)$.
\item $\mcM$ is the invertible sheaf on $X$ associated to the triple cover $\tau$.
\end{itemize}
\end{Setting}

We set 
\[
U_x = (x \ne 0) \subset Z, \ 
\Pi_y = (x = 0) \cap (y \ne 0) \subset Z, \ 
\Pi_z = (x = y = 0) \subset Z,
\]
so that we have $Z = U_x \cup \Pi_y \cup \Pi_z$.

\subsubsection{Subcase $\nu > \mu$}

\begin{Lem}
The section $f \in H^0 (Z, \mcO_Z (\theta,3))$ has only nondegenerate critical points on $Z$.
\end{Lem}

\begin{proof}
This follows from Lemma \ref{lem:surjrest} since $\theta \ge \max \{2, 2 \lambda, 2 \mu\}$ and $\theta \ge 3 \mu + 1$.
\end{proof}

\subsubsection{Subcase $\nu = \mu > \lambda$}

\begin{Lem}
The section $f \in H^0 (Z, \mcL^3)$ has only nondegenerate critical points on $Z \setminus \Gamma_z$.
\end{Lem}

\begin{proof}
This follows from Lemma \ref{lem:surjrest} since $\theta = 3 \nu \ge \max \{2, 2 \lambda, 2 \mu\}$ and $\theta = 3 \nu \ge \max \{3 \lambda + 1, 2 \lambda + \mu \}$.
\end{proof} 

We can write
\[
f = \alpha z^3 + a_{\mu} z^2 x + b_{\mu - \lambda} z^2 y + g (u,x,y,z),
\]
where $\alpha \in \K$, $a_{\mu}, b_{\mu - \lambda}$ are polynomials in $u$ of indicated degree and $g \in (x,y)^2$.
Let $\Crit (f)$ be the set of critical points of $f$.

\begin{Lem}
We have $\Crit (f) = C_1 \cup C_2$, where $C_1$ consists of nondegenerate critical points on $Z \setminus \Gamma_z$ and $C_2 = (x = y = a = b = 0)$.
\end{Lem}

\begin{proof}
Set $C_1 = \Crit (f) \cap (Z \setminus \Gamma_z)$ and $C_2 = \Crit (f) \cap \Gamma_z$ so that $\Crit (f) = C_1 \cup C_2$.
Then it is easy to see that $C_2 = (x = y = a = b = 0)$.
\end{proof}

The set $C_2$ is covered by the $U_{i,z} = (u_i \ne 0) \cap (z \ne 0) \subset Z$ for $i = 0,\dots,n-2$.
We analyze the restriction of $f$ to $U_{0,z}$.
The analysis for the other $U_{i,z}$ is completely the same.
The open set $U_{0,z}$ is isomorphic to $\mbA^n$ with affine coordinates $u_1,\dots,u_{n-2},x,y$ and we have
\[
f \sim_3 \bar{a} x + \bar{b} y + \bar{g},
\]
where $\bar{g} = g (1,u_1,\dots,u_{n-2},x,y,1)$ and similarly for $\bar{a}$ and $\bar{b}$.
We have $\deg a = \mu > 0$, $\deg b = \mu - \lambda > 0$, and the complete intersection in $\mbA^{n-2}_u$ defined by $\bar{a} = \bar{b} = 0$ is nonsingular since $a,b$ are general.
Moreover $\bar{g} \in (x,y)^2$.
Thus, by Lemma \ref{lem:resol2}, the blowup of $X$ along $\tau^{-1} (C_2)$ is universally $\CH_0$-trivial, resolves the singularity of $X$ along $\tau^{-1} (C_2)$ and pulls back $\mcM$ into a subsheaf of $\Omega^{n-1}$.

\subsubsection{Subcase $\nu = \mu = \lambda$}

\begin{Lem}
The section $f \in H^0 (Z, \mcL^3)$ has only nondegenerate critical points on $U_x$.
\end{Lem}

\begin{proof}
This follows from Lemma \ref{lem:surjrest} since $\theta \ge \max \{2, 2 \lambda, 2 \mu\}$.
\end{proof} 

We can write
\[
f = \alpha y^3 + \beta y^2 z + \gamma y z^2 + \delta z^3 + x (a_{\mu} y^2  + b_{\mu} y z + c_{\mu} z^2) + g (u,x,y,z),
\]
where $\alpha,\beta,\gamma,\delta \in \K$, $a_{\mu}, b_{\mu}, c_{\mu}$ are homogeneous polynomials in $u$ of degree $\mu$ and $g \in (x)^2$.
Let $\Crit (f)$ be the set of critical points of $f$.
Replacing $z$ and rescaling $y$, we may assume $\gamma = 0$ and $\beta = 1$.

\begin{Lem}
We have $\Crit (f) = C_1 \cup C_2$, where $C_1$ consists of nondegenerate critical points on $Z \setminus \Gamma_z$ and $C_2 = (x = y = c = 0)$.
\end{Lem}

\begin{proof}
Set $C_1 = \Crit (f) \cap U_x$ and $C_2 = \Crit (f) \cap (Z \setminus U_x)$ so that $\Crit (f) = C_1 \cup C_2$.
Then it is easy to see that $C_2 = (x = y = c = 0)$.
\end{proof}

The set $C_2$ is covered by $U_{i,z} = (u_i \ne 0) \cap (z \ne 0) \subset Z$ for $i = 0,\dots,n-2$.
We analyze the restriction of $f$ to $U_{0,z}$.
The analysis for the other $U_{i,z}$ is completely the same.
We write $g = d_{2 \mu} x^2 z + h$, where $d_{3 \mu} = d_{3 \mu} (u)$ and $h = h (u,x,y,z)$ does not contain a monomial divisible by $x^2 z$.
Note that $h \in (x,y)^3$.
The open subset $U_{0,z}$ is isomorphic to $\mbA^n$ with affine coordinates $u_1,\dots,u_{n-2},x,y$ and we have
\[
f \sim_3 \bar{c} x + \bar{d} x^2 + \bar{b} x y + y^2 + (\bar{a} x y^2 + \bar{h}), 
\]
where $\bar{g} = g (1,u_1,\dots,u_{n-2},x,y,1)$ and similarly for the others.
We have $\deg (\bar{c}) > 0$ and the hypersurface in $\mbA^{n-2}_u$ defined by $\bar{c} = 0$ is nonsingular since $c$ is general.
Moreover we have $\bar{a} x y^2 + \bar{h} \in (x,y)^3$.
Thus, by Lemma \ref{lem:resol2}, the blowup of $X$ along $\tau^{-1} (C_2)$ is universally $\CH_0$-trivial, resolves the singularity of $X$ along $\tau^{-1} (C_2)$ and pulls back $\mcM$ into a subsheaf of $\Omega^{n-1}$.

\subsection{Case $\theta < 3 \nu$} \label{sec:dP3case3}

In this case we have $\theta \ge 2 \nu$ and we keep the following setting.

\begin{Setting} \label{setting:dP3-3}
\begin{itemize}
\item The ground field $\K$ is algebraically closed and $\chara (\K) = 2$.
\item $X$ is a hypersurface $X$ in $P = P_3 (n,\lambda,\mu,\nu)$ defined by
\[
F := w^2 (a_{\theta - 2 \nu} x + b_{\theta - 2 \nu - \lambda} y + c_{\theta - 2 \nu - \mu} z) + f (u,x,y,z) = 0,
\]
where $a,b,c$ are homogeneous polynomials of the indicated degree in the variable $u$ and $f (u,x,y,z)$ is of bi-degree $(\theta,3)$.
We assume that $a,b,c$ and $f$ are general.
\item $R'$ is the $\mbP (1,1,1,2)$-bundle over $\mbP^{n-2}$ defined by
\[
\begin{pmatrix}
u_0 & \cdots & u_{n-2} & & x & y & z & \bar{w} \\
1 & \cdots & 1 & | & 0 & \lambda & \mu & 2 \nu \\
0 & \cdots & 0 & | & 1 & 1 & 1 & 2
\end{pmatrix}
\]
\item $Z$ is the hypersurface in $R'$ defined by $\bar{F} := \bar{w} (a x + b y + c z) + f = 0$.
\item $\mcL = \mcO_Z (\nu,1)$.
\item $\tau \colon X \to Z$ is the restriction of the double cover $P \to R'$, which is the double cover of $Z$ branched along $\bar{w} \in H^0 (Z,\mcL^2)$.
\item $Z^{\circ} = Z \setminus (x = y = z = 0) \cap Z$, $X^{\circ} = \tau^{-1} (Z^{\circ})$ and $\tau^{\circ} = \tau|_{X^{\circ}}$. 
\end{itemize}
\end{Setting}

We set $g = a x + b y + c z$.
We understand, for example, that $c = 0$ when $\theta - 2 \nu - \mu < 0$.

\begin{Rem} \label{rem:commonsol}
By (3), (4) and (5) of Lemma \ref{lem:dP3num}, the system of equations $a = b = c = 0$ does not have a solution in $\mbP^{n-2}$.
\end{Rem}

\begin{Lem}
$Z^{\circ}$ is nonsingular and $X$ is nonsingular along $X \setminus X^{\circ}$.
\end{Lem}

\begin{proof}
Let $Q$ be the $\mbP^2$-bundle over $\mbP^{n-2}$ defined in the proof of Lemma \ref{lem:nonsingZdP3}.
If $\theta = 2 \nu$, then $Z \cong Q$, so that $Z^{\circ} = Z$ is nonsingular and $X = X^{\circ}$.
In the following we assume that $\theta > 2 \nu$.
For the first assertion, it is enough to show that $\Delta := (\bar{w} = g = f = 0) \subset R'$ is nonsingular by Lemma \ref{lem:crismhyp}.
We identify $Q$ with $(\bar{w} = 0) \subset R'$. 
The hypersurface $H = (g = 0) \subset Q$ is nonsingular since $g = a x + b y + c z$ is linear with respect to $x,y,z$ and $a = b = c = 0$ has no non-trivial solution (see Remark \ref{rem:commonsol}).
The section $f$ can be viewed as a general element of $H^0 (Q,\mcO_{Q} (\theta,3))$ and $\mcO_{Q} (\theta,3)$ is very ample since $\theta > \mu$.
Now $\Delta$ isomorphic to the hypersurface of $H$ cut out by $f = 0$, which is nonsingular by Bertini theorem.

We prove the second assertion.
We have 
\[
\Xi := X \setminus X^{\circ} = (x = y = z = 0) \subset P.
\]
and 
\[
\frac{\prt F}{\prt x}|_{\Xi} = a w^2, \ 
\frac{\prt F}{\prt y}|_{\Xi} = b w^2, \ 
\frac{\prt F}{\prt z}|_{\Xi} = c w^2.
\]
From this we deduce that $X$ is nonsingualr along $\Xi = X \setminus X^{\circ}$ since the locus $(a = b = c = 0) \subset P$ is empty.
\end{proof}

Since $Z^{\circ}$ is nonsinguar, we can define $\mcM^{\circ}$ to be the invertible sheaf on $X^{\circ}$ associated to the double cover $\tau^{\circ}$, and we set $\mcM = \iota_*\mcM^{\circ}$, where $\iota \colon X^{\circ} \inj X$.

\begin{Lem}
The section $\bar{w} \in H^0 (Z, \mcL^2)$ has only $($almost$)$ nondegenerate critical points on $Z^{\circ}$.
\end{Lem}

\begin{proof}
As in the proof of Lemma \ref{lem:critdP1case1}, it is enough to show that the section $g f \ (= - \bar{w} g^2)$, which is viewed as a section of $\mcO_{Q} (2 \theta - 2 \nu ,4)$, has only (almost) nondegenerate critical points on $(g \ne 0) \subset Q$.
This follows if we show that the 4th (resp.\ 2nd) restriction map $r_4 (\msp)$ of $\mcO_{Q} (\theta,3)$ is surjective at any point of $U_x \cap (g \ne 0)$ (resp.\ $(\Pi_y \cup \Gamma_z) \cap (g \ne 0)$).

We first consider the case $\theta \ge 3$.
If in addition $\theta > 3 \mu$, then the assertion follows from Lemma \ref{lem:surjrest}.
In the following we assume $\theta = 3 \mu$.
Then, since $\theta < 2 \nu + \mu$, we are in one of the following cases:
\begin{enumerate}
\item[(i)] $n = 3$ and $2 \nu + \lambda < \theta = 3 \mu < 2 \nu + \mu$.
\item[(ii)] $\theta = 3 \mu = 2 \nu + \lambda$.
\item[(iii)] $\theta = 3 \mu = 2 \nu$.
\end{enumerate} 

In any of the above cases, we have $\theta \ge \max \{3, 3 \lambda, 3 \mu\}$ and, by Lemma \ref{lem:surjrest}, $r_4 (\msp)$ is surjective at any point $\msp \in U_x$ for a general $f \in H^0 (\mcO_{Q} (\theta,3))$.
If we are in case (iii), then the proof is completed since we may assume $g = x$ in this case and thus $(g \ne 0) = U_x$. 

Suppose that we are in case (i) or (ii).
We have $\theta \ge \max \{3 \lambda + 1, 2 \lambda + \mu \}$ and, by Lemma \ref{lem:surjrest}, $r_2 (\msp)$ is surjective at any point $\msp \in \Pi_y$.
It follows that $g f $ does not have a critical points along $\Pi_y$.
Now, since $g = a x + b y$, the set $(g \ne 0) \subset Q$ is contained in $U_x \cup \Pi_y$.
Thus $g f$ has only (almost) nondegenerate critical points on $(g \ne 0) \subset Q$.

We consider the case when $\theta < 3$.
In this case $(\theta,\lambda,\mu,\nu) = (2,0,0,1)$ and $a, b, c$ are constants.
Thus we may assume $g = x$ and it is enough to show that the section $x f$ has only (almost) nondegenerate critical points on $U_x$.
We have $\theta \ge \max \{2,2 \lambda, 2\mu\}$, which implies surjectiveity of $r_3 (\msp)$ for $\msp \in U_x$.
Thus, if $n$ is even, then $x f$ has only nondegenerate critical points on $U_x$.
It remains to consider the case when $n$ is odd.
Let $\msp \in U_x$ be a point.
It is clear that $r_2 (\msp)$ is surjective.
It follows that the sections $f \in H^0 (Q,\mcO_Q (2,3))$ such that $x f$ has a critical point at $\msp$ form a codimension $n$ subspace of $H^0 (Q,\mcO_Q (2,3))$.
Thus it is enough to show that the existence of $f \in H^0 (Q,\mcO_Q (2,3))$ which has an almost nondegenerate critical point at $\msp$.
We may assume that $\msp = (1\!:\!0\!:\!\cdots\!:\!0 ; 1\!:\!0\!:\!0)$ and consider a section
\[
f = (u_1 u_2 + \cdots + u_{n-4} u_{n-3}) x^3 + u_0 u_{n-2} x^2 y + u_0^2 z^3 + \cdots \in H^0 (Q, \mcO_Q (2,3)).
\]
It is easy to see that $f$ has almost nondegenerate critical point at $\msp$ and the proof is completed.
\end{proof}

We summarize the results of Sections \ref{sec:dP3case1}, \ref{sec:dP3case2} and \ref{sec:dP3case3}.
 
\begin{Prop} \label{prop:dP3goodresol}
Suppose that $(\theta,\lambda,\mu,\nu)$ satisfies the assumption of \emph{Lemma \ref{lem:dP3num}} and that $(\theta,\lambda,\mu,\nu) \ne (1,0,0,0)$.
Let $X$ be as in one of \emph{Settings \ref{setting:dP3-1}, \ref{setting:dP3-2}, \ref{setting:dP3-3}}.
Then there exists a universally $\CH_0$-trivial resolution $\varphi \colon Y \to X$ of singularities of $X$ such that $\varphi^*\mcM \inj \Omega_Y^{n-1}$, and we have
\[
\mcM \cong \mcO_X (\theta - \lambda - \mu - (n-1), 0).
\]
\end{Prop}

\subsection{Proofs of Theorem \ref{thm:dP3} and Corollary \ref{cor:dP3}}

\begin{proof}[Proof of \emph{Theorem \ref{thm:dP3}}]
This follows from Proposition \ref{prop:dP3goodresol} and Theorem \ref{thm:spCH0}.
\end{proof}

\begin{proof}[Proof of \emph{Corollary \ref{cor:dP3}}]
This follows from Lemma \ref{lem:dP3ample} and Theorem \ref{thm:dP3}.
\end{proof}

We prove the results in Section \ref{sec:intro}.
Theorems \ref{mainthm1} and \ref{mainthm2} follow from Corollaries \ref{cor:dP1}, \ref{cor:dP2} and \ref{cor:dP3}.
We prove Theorem \ref{mainthm3} in which the variety in (1) corresponds to $(\lambda,\mu,\nu) = (0,0,m)$ in Theorem \ref{thm:dP2} and those in (2), (3) correspond to $(\lambda,\mu,\nu,\theta) = (0,0,m,3m)$, $(0,0,0,d)$ in Theorem \ref{thm:dP3}, respectively.
Therefore Theorem \ref{mainthm3} follows immediately.

\end{document}